\documentclass[12pt,reqno]{article}
\usepackage{amssymb}
\usepackage{enumitem}
\usepackage{bm}
\usepackage{lscape}
\usepackage{booktabs,caption}
\usepackage{multirow}
\usepackage{graphicx}
\usepackage{amsmath,amsthm}
\usepackage{a4wide}
\usepackage{version}
\usepackage{parskip}
\usepackage{euscript}
\usepackage{mathrsfs}
\usepackage{tikz}
\usepackage{mathtools}
\usetikzlibrary{arrows,decorations.markings, matrix}
\usepackage{hyperref}
\usepackage{comment}

\usepackage[all]{xy}

\numberwithin{equation}{section}

\newtheorem{thm}{Theorem}[section]
\newtheorem{lem}[thm]{Lemma}

\newtheorem{cor}[thm]{Corollary}
\newtheorem{conj}[thm]{Conjecture}

\theoremstyle{definition}
\newtheorem{definition}[thm]{Definition}
\newtheorem{exm}[thm]{Example}
\newtheorem{rmk}[thm]{Remark}

\newcommand{\N}{\mathbb{N}}
\newcommand{\Z}{\mathbb{Z}}
\renewcommand{\k}{\mathbf{k}}

\newcommand{\g}{\mathfrak{g}}
\newcommand{\h}{\mathfrak{h}}
\newcommand{\gl}{\mathfrak{gl}}

\newcommand{\FF}{{\mathcal F}}
\newcommand{\WW}{{\mathcal W}}

\newcommand{\MM}{{\mathcal M}}

\renewcommand{\AA}{{\mathcal A}}
\newcommand{\HH}{\operatorname{HH}}

\renewcommand{\SS}{{\mathcal S}}

\newcommand{\C}{{\mathbb C}}

\newcommand{\Q}{{\mathbb Q}}

\newcommand{\SH}{\operatorname{SH}}

\renewcommand{\k}{\mathbf{k}}

\renewcommand{\mod}{\operatorname{mod}}

\newcommand{\BB}{{\cal B}}

\newcommand{\G}{{\mathbb G}}

\renewcommand{\ker}{\operatorname{ker}}

\newcommand{\A}{{\mathbb A}}

\newcommand{\lcm}{\operatorname{lcm}}

\newcommand{\OP}{\operatorname}

\newcommand{\w}{{\bf w}}
\newcommand{\mf}{\operatorname{mf}}
\renewcommand{\mho}{T}
\newcommand{\eu}{\mathfrak{eu}}
\newcommand{\lle}{l}
\newcommand{\ohr}{\nu}
\newcommand{\enn}{\ell} 

\makeatletter
\def\thm@space@setup{%
  \thm@preskip=\parskip \thm@postskip=0pt
}
\makeatother

\newcommand\disccap[2]{
  \begin{scope}[shift={(#1,#2)}]
    \draw[dotted] (0,0) circle [x radius = 0.5, y radius = 0.25];
    \draw[dashed] (0.5,0) arc [x radius = 0.5, y radius = 0.75, start angle = 0, end angle = -180];
  \end{scope};
}

\newcommand\cylind[2]{
  \begin{scope}[shift={(#1,#2)}]
    \draw[dotted] (0,0) circle [x radius = 0.5, y radius = 0.25];
    \draw[dotted] (2,0) circle [x radius = 0.5, y radius = 0.25];
    \draw[dashed] (2.5,0) arc [x radius = 1.5, y radius = 1.25, start angle = 0, end angle = -180];
    \draw[dashed] (1.5,0) arc [x radius = 0.5, y radius = 0.5, start angle = 0, end angle = -180];    
  \end{scope};
}

\title{Symplectic cohomology of\\ compound Du Val singularities}
\author{Jonny Evans\hspace{2cm} Yank\i\ Lekili}

\newcommand{\Addresses}{{
  \bigskip
  \footnotesize
  J.~D.~Evans, \textsc{University of Lancaster}\par\nopagebreak
  \texttt{j.d.evans@lancaster.ac.uk}

  \medskip

  Y.~Lekili, \textsc{Imperial College London}\par\nopagebreak
  \texttt{y.lekili@imperial.ac.uk}

}}

\date{} 

\begin{document}

\maketitle 

\begin{abstract}
  We compute symplectic cohomology for Milnor fibres of certain
  compound Du Val singularities that admit small resolution by using
  homological mirror symmetry. Our computations suggest a new
  conjecture that the existence of a small resolution has strong
  implications for the symplectic cohomology and conversely. We also
  use our computations to give a contact invariant of the link of the
  singularities and thereby distinguish many contact structures on
  connected sums of \(S^2\times S^3\).
\end{abstract}

\section{Introduction}

\subsection{Links}

Let \(X\subset\C^N\) be a normal \(n\)-dimensional algebraic variety
over \(\C\) and let \(P\in X\) be a point; we will write \([P\in X]\)
for the germ of \(X\) at \(P\) considered up to local analytic
equivalence. Recall that the {\em link} of \(P\in X\), written
\(\OP{Link}(P)\), is the intersection of a small Euclidean sphere
centred at \(P\) with \(X\). If \(P\) is a smooth point or isolated
singularity then the link is a smooth, compact \((2n-1)\)-dimensional
manifold; we will focus on hypersurface singularities, whose link is
\((n-2)\)-connected. How much information do we retain about
\([P\in X]\) if we only remember the manifold \(\OP{Link}(P)\)?

Mumford \cite{Mumford} proves that if \(n=2\) then \(\OP{Link}(P)\) is
a simply-connected 3-manifold if and only if \(P\in X\) is a smooth
point. By contrast, in higher dimensions, the topology of the link
exerts less influence. For example, if \(\Sigma\) is any homotopy
7-sphere, Brieskorn \cite{BrieskornLink2} constructs singular complex
4-folds \(P_k\in X_k\), \(k\in\N\), with
\([P_i\in X_i]\neq [P_j\in X_j]\) for \(i\neq j\) and
\(\OP{Link}(P_k)\cong\Sigma\). More generally, when \(n\geq 3\), surgery
theory tells us there are not very many \((n-2)\)-connected
\((2n-1)\)-manifolds\footnote{For a classification, see Wall
  \cite{Wall}.}, but there are lots of singularities.

The field of complex tangencies \(\xi\) forms a contact distribution
on \(\OP{Link}(P)\) \cite{Varchenko}. McLean \cite{McLean}
demonstrates that the contact manifold \((\OP{Link}(P),\xi)\) retains
much more information about \([P\in X]\). For example, he shows that
\((\OP{Link}(P),\xi)\) is contactomorphic to the standard contact
5-sphere if and only if \(P\in X\) is a smooth point, and that the
minimal discrepancy of a canonical \(\Q\)-Gorenstein singularity
\(P\in X\) is determined by \((\OP{Link}(P),\xi)\).

An interesting corollary of McLean's work relates the purely
algebro-geometric notion of {\em terminal} singularities to the purely
contact geometric notion of {\em dynamical convexity}.
\begin{itemize}
\item A singularity is called terminal if its minimal discrepancy is
  positive. Terminal singularities emerged in the work of Reid
  \cite{Reid} as a natural class of singularities that should appear
  on minimal models of smooth 3-folds. The 3-fold terminal
  singularities were classified by Mori \cite{Mori}.
\item A Reeb flow on a contact manifold is called {\em dynamically
    convex} if every closed Reeb orbit \(\gamma\) satisfies
  \(\mu_{CZ}(\gamma)+n-3 > 0\), where \(\mu_{CZ}\) is the
  Conley-Zehnder index. A contact manifold which admits a dynamically
  convex Reeb flow is called {\em index positive}.
\end{itemize}

\begin{thm}[\cite{McLean}]\label{cor:mclean}
  Suppose that \(P\in X\) is an isolated \(\Q\)-Gorenstein singularity
  with \(H^1(\OP{Link}(P);\Q)=0\) (e.g. a hypersurface singularity of
  dimension \(n\geq 3\)). The singularity \(P\in X\) is terminal if
  and only if its link \((\OP{Link}(P),\xi)\) is index positive.
\end{thm}
\begin{proof}
  If \(P\in X\) is terminal then the minimal discrepancy is positive,
  so McLean's theorem implies that the highest minimal SFT index of
  the link is positive, which is precisely the statement that there is
  a dynamically convex Reeb flow on \((\OP{Link}(P),\xi)\). If there
  is no dynamically convex Reeb flow then the highest minimal SFT
  index is nonpositive so, by McLean's theorem, the minimal
  discrepancy is also nonpositive; therefore \(P\in X\) is not
  terminal.
\end{proof}

Invariants of contact manifolds (like contact homology or
symplectic field theory) are notoriously difficult to define
because of bubbling of pseudoholomorphic curves in
symplectisations. The condition of index positivity allows us to
bypass many of these problems to get useful contact invariants.
For example, if \(Y\) is a contact manifold and \(V\) is a
simply-connected strong symplectic filling of \(Y\) with
$c_1(V)=0$, then we can define symplectic cohomology
\[ \SH^*(V;\mathbb{C}) \] as a $\mathbb{Z}$-graded\footnote{More
  generally, the possible \(\Z\)-gradings on
  \( \SH^*(V;\mathbb{C}) \) form a torsor over
  \( \mathrm{H}^1(V;\Z)\). Note that with our grading
  conventions an orbit with Conley-Zehnder index \(\mu\) lives
  in degree \(n-\mu\) where \(2n=\dim V\). In particular the
  unit lives in degree zero and a constant orbit corresponding
  to a critical point of Morse index \(k\) lives in degree
  \(k\).}  $\mathbb{C}$-vector space (with various additional
algebraic structures on it) which usually depends on
\(V\). However, if \(Y\) is simply-connected and index positive
then the dependence of $\SH^*(V;\mathbb{C})$ on the filling is
very mild. The positive symplectic cohomology
$\SH^*_+(V;\mathbb{C})$, constructed as the cohomology of a
quotient complex of the cochain complex of $\SH^*(V;\mathbb{C})$
by the cochains coming from the interior of the filling, is
known to be a contact invariant {\cite[Proposition
  9.17]{CieliebakOancea}}. This has been used successfully by
Uebele to distinguish some contact structures on
\(S^2\times S^3\) \cite{Uebele}.

We explore a refinement of this in Corollary \ref{cor:lie_alg}. In
particular, for $n=3$, we are able prove by a standard neck-strecthing
technique that the Lie algebra structure on $\SH^1(V;\mathbb{C})$ and
its Lie algebra representation on
$\bigoplus_{d<0} \SH^d(V;\mathbb{C})$ is a contact invariant.

Our goal in this paper is to compute symplectic cohomology
for some further examples of links of terminal 3-fold hypersurface
singularities, observe some patterns which emerge, and use it to
distinguish a variety of links.

\subsection{Compound Du Val (cDV) singularities}

It is a theorem of Reid {\cite[Theorem 1.1]{Reid}} that the Gorenstein
terminal 3-fold singularities are precisely the isolated {\em compound
  Du Val (cDV) singularities}. These are hypersurface singularities
which (in suitable local analytic coordinates \((w,x,y,z)\)) are cut
out by an equation of the form
\[f(x,y,z)+wg(x,y,z,w)=0\] where \(f\) is one of the following
polynomials:
\begin{align*}
  A_\enn: &\quad x^2 + y^2 + z^{\enn + 1},\\
  D_\enn: &\quad x^2 + y (z^2 + y^{\enn - 2}),\\
  E_6: &\quad x^2 + y^3 + z^4,\\
  E_7: &\quad x^2 + y (y^2 + z^3),\\
  E_8: &\quad x^2 + y^3 + z^5.
\end{align*}
The \(w=0\) hyperplane section has an ADE singularity at \(0\). If
\(\Gamma\) is the ADE type of this hyperplane section, we refer to the
3-fold singularity as a {\em compound \(\Gamma\)} or {\em \(c\Gamma\)}
singularity.

As we have explained in Theorem \ref{cor:mclean}, the links of these
singularities are index positive and so we can use $\SH^*$ of the
Milnor fibre for $*<0$ as a contact invariant.

\begin{rmk}\label{rmk:milnor_obd}
  Observe that if we define \(B\subset \OP{Link}(0)\) to be the
  intersection \(\{w=0\}\cap\OP{Link}(0)\) then we get a Milnor open
  book
  \[w/|w|\colon\OP{Link}(0)\setminus B\to S^1\] with binding \(B\).
  The page is a copy of the corresponding 4-dimensional ADE Milnor
  fibre and the contact structure determined by the open book is
  contactomorphic to \(\xi\).
\end{rmk}

\begin{exm}\label{exm:ca1}
  Consider the family of \(cA_1\)
  singularities
  \[A_{\enn}:=\{x^2+y^2+z^2+w^{\enn+1}=0\},\qquad\enn\geq 1.\] In
  fact, any \(cA_1\) singularity is equivalent to one of these. The
  link is either \(S^5\) (if \(\enn\) is even) or \(S^2\times S^3\)
  (if \(\enn\) is odd). The page of the Milnor open book is the
  \(A_1\)-Milnor fibre \(T^*S^2\), and the monodromy is the
  \((\enn+1)\)st power of a Dehn twist in the zero-section. The
  symplectic cohomology of the Milnor fibre \(V_{\enn}\) behaves
  differently if \(\enn\) is odd or even. If \(\enn\) is even then, by
  {\cite[Section 5.2]{LU2}}, we have
  \[\SH^*(V_\enn;\C)=
    \begin{cases}
      \C^{\enn}&\mbox{ if }*=3\\
      \C&\mbox{ if }*=-q(\enn+3)-r\mbox{ for
            }r\in\{0,\ldots,\enn-1\},\  r = q (\text{mod} 2) \\
      \C&\mbox{ if }*=-q(\enn+3)-r+1\mbox{ for
            }r\in\{0,\ldots,\enn-1\},\  r = q (\text{mod} 2)  \\
      0&\mbox{ otherwise,}
    \end{cases}\] for $q \in \mathbb{N}$. In particular, we see that $\SH^*$ can be either 0 or $\mathbb{C}$ for $*<0$. 

  If \(\enn\) is odd then we will see below that
  \[\SH^*(V_\enn;\C)=
    \begin{cases}
      \C^{\ell}&\mbox{ if }*=3\\
      \C&\mbox{ if }*=1\mbox{ or }*<0\\
      0&\mbox{ otherwise}.
    \end{cases}\] Write \(\xi_\enn\) for the contact structure on the
  link of \(0\in A_\enn\). Since the contact invariant \(\SH^*_+\)
  coincides with \(\SH^*\) if \(*<0\), this shows that
  \begin{enumerate}
  \item[(a)] the links
    \(\{(S^5,\xi_{\enn})\ :\ \enn = 2,4,6,\ldots\}\) are pairwise
    nonisomorphic as contact manifolds,
  \item[(b)] we cannot distinguish the
    links \(\{(S^2\times S^3,\xi_\enn)\ :\ \enn = 1,3,5,\ldots\}\) using
    \(\SH^*_+\) with coefficients in \(\C\).
  \end{enumerate}
  A similar phenomenon was observed by Van Koert {\cite[Example
    3.1.1]{VanKoert}} for these contact structures on
  \(S^2\times S^3\): they are not distinguished by their cylindrical
  contact homology. Interestingly, Uebele \cite{Uebele} does
  distinguish them using \(\SH^*_+\) with coefficients in \(\Z/2\). We
  will give a second way to distinguish them below.
\end{exm}

From an algebro-geometric perspective, the singularities \(A_\enn\)
have different behaviour when \(\enn\) is even/odd. For example, these
singularities admit {\em small resolutions}\footnote{Recall that a
  small resolution is a resolution whose exceptional set has
  codimension at least 2. Note that, by {\cite[Theorem 1.14]{Reid}}, a
  resolution of an isolated cDV singularity is small if and only if it
  is crepant.} if and only if \(\enn\) is odd; indeed, if \(\enn\) is
odd, there is a resolution whose exceptional set is an irreducible
rational curve. However, if \(\enn\) is even then there cannot be a
small resolution because the link is not diffeomorphic to a nontrivial
connected sum of copies of \(S^2\times S^3\).

Inspired by this example, we record an optimistic conjecture, which
provides the main motivation for the calculations in this paper. We
will establish this conjecture in a range of examples (Theorem
\ref{thm:summary}).

\begin{conj}\label{conj:small_res}
  Suppose that \(P\in X\) is a cDV singularity and let \(V\) be the
  Milnor fibre of the singularity. Then \(P\in X\) admits a small
  resolution such that the exceptional set has \(\enn\) irreducible
  components if and only if \(\SH^*(V;\C)\) has rank \(\enn\) in every
  negative degree.
\end{conj}

\begin{rmk}
  In this paper, we have focused on providing evidence for one
  direction of this conjecture: that the existence of a small
  resolution constrains the symplectic cohomology. The converse is
  plausible: we have calculated many examples and found no
  counterexample. This would give an a priori way of detecting whether
  a cDV singularity admits a small resolution just by knowing its
  link.
\end{rmk}

\begin{rmk}\label{rmk:versal}
  If \(P\in X\) admits a small resolution whose exceptional set has
  \(\enn\) irreducible components then the link is diffeomorphic to
  \(\sharp_{\enn}(S^2\times S^3)\). Small resolutions can be
  constructed by thinking of the 3-fold as a 1-parameter deformation
  of an ADE singularity. This gives a classifying map from the disc to
  the versal deformation space of the ADE singularity such that the
  3-fold is the pullback along the classifying map of the versal
  family. Brieskorn \cite{Brieskorn}, Tjurina \cite{Tjurina} and
  Pinkham \cite{Pinkham} constructed branched coverings of the versal
  ADE deformation space (branched over the discriminant locus) such
  that the pullback of the versal family to the branched covering
  admits a simultaneous (partial) resolution. More precisely, the
  fundamental group of the complement of the discriminant locus is the
  ADE Artin braid group; Brieskorn and Tyurina constructed the
  branched covering corresponding to the kernel of the homomorphism to
  the ADE Weyl group and found a full simultaneous resolution, while
  Pinkham constructed simultaneous partial resolutions for
  intermediate covering spaces. For a specific 3-fold, if the
  classifying map from the disc lifts (in the sense of algebraic
  topology) to one of these branched covers, then you get a small
  resolution by pulling back the simultaneous partial resolution of
  the versal family. In particular, the existence of a small
  resolution can be read off from the monodromy of the Milnor open
  book mentioned in Remark \ref{rmk:milnor_obd} (which is the element
  of the fundamental group of the ADE Artin braid group represented by
  the boundary of the disc under the classifying map).
\end{rmk}

\begin{rmk}
  Remark \ref{rmk:versal} provides a sanity check on Conjecture
  \ref{conj:small_res}. Consider what happens if we deform the germ of
  the singularity at \(P\). Namely, suppose we have a family
  \(h_s(w,x,y,z)=f(x,y,z)+g_s(w,x,y,z)\) of cDV singularities
  parametrised by \(s\in\C\). Suppose that there are balls
  \(B\subset\C\) and \(B'\subset\C^4\) such that for \(s\in B\), the
  origin is the only singularity of the hypersurface
  \(h_s^{-1}(0)\cap B'\). Gray's stability theorem tells us that the
  contact geometry of the link of the singularity is independent of
  \(s\in B\). Moreover, if \(0\in h_0^{-1}(0)\) admits a small
  resolution then so do all the singularities \(0\in h_s^{-1}(0)\)
  because the monodromy of the Milnor open book is stable under
  perturbations.
\end{rmk}

We now summarise our evidence for Conjecture
\ref{conj:small_res}. These calculations will be explained in
Section \ref{sct:compendium}. Throughout, we work over $\mathbb{C}$.

\begin{thm}\label{thm:summary}
  The table below summarises our calculations of symplectic cohomology
  for Milnor fibres of some cDV singularities. The left-most column is
  a polynomial \(\check{\w}\) and the singularity is defined by
  \(0\in \check{\w}^{-1}(0)\). The columns \(\SH^*\) give the ranks of
  the various graded pieces of \(\SH(\check{\w}^{-1}(1))\). In all
  cases, \(\SH^d(\check{\w}^{-1}(1))=0\) if \(d=2\) or \(d\geq
  4\). The final column gives a reference for the calculation. Case 4
  is conditional on Conjecture \ref{conj:BH} or Conjecture
  \ref{conj:LU}, so we have marked it with an asterisk.

  \begin{tabular}{llcccl}
    \toprule
    &Singularity                                    & ADE type & \(\SH^3\)          & \(\SH^{d\leq 1}\) & See Theorem...\\
    \midrule
    1.& \(x_1^2+x_2^2+x_3^{\enn+1}+x_4^{k(\enn+1)}\)          & \(A_\enn\)  & \(\enn(k(\enn+1)-1)\)   &  \(\enn\) & \ref{thm:bp} (1), (2)\\
    2.& \(x_1^2+x_2^2+x_3x_4(x_3^{\enn-1}+x_4^{k(\enn-1)})\)  & \(A_\enn\)  & \((k\enn+1)(\enn-1)\)   & \(\enn\) & \ref{thm:can}\\
    3.& \(x_1^2+x_2^3+x_3^3+x_4^{6k}\)                & \(D_4\)  & \(24k-4\)         & \(4\) & \ref{thm:bp} (3)\\
    4*.& \(x_1^3+x_1x_2^{2k+1}+x_2x_3^2+x_4^2\)         & \(D_4\)  & \(6k+5\)          & \(1\) & \ref{thm:laufer}\\
    5.& \(x_1^2+x_2^3+x_3^4+x_4^{12k}\)               & \(E_6\)  & \(72k-6\)         & \(6\) & \ref{thm:bp} (4)\\
    6.& \(x_1^2+x_2^3+x_3^5+x_4^{30k}\)               & \(E_8\)  & \(240k-8\)        & \(8\) & \ref{thm:bp} (5)\\
    \bottomrule
  \end{tabular}
\end{thm}
\begin{rmk}
  In all cases, these singularities admit small resolutions and the
  number of exceptional curves in the resolution equals the rank of
  \(\SH^d\) for \(d\leq 1\); this is explained case-by-case in Section
  \ref{sct:compendium}. In particular, this establishes Conjecture
  \ref{conj:small_res} for these examples.
\end{rmk}

\begin{rmk}
  The examples in Theorem \ref{thm:summary} are all {\em invertible
    polynomials} (see Section \ref{sct:invertible}), and our strategy
  for calculating symplectic cohomology uses mirror symmetry for
  invertible polynomials to relate \(\SH\) with the Hochschild
  cohomology of a mirror dg-category of equivariant matrix
  factorisations. In all cases except case 4, the required mirror
  symmetry conjecture is proven. Case 4 is only proved conditionally
  (see Section \ref{sct:conjectures}). This example is the base of the
  Laufer flop \cite{Laufer}.
\end{rmk}

\begin{rmk}
  Theorem \ref{thm:summary} seems to indicate that symplectic
  cohomology (over $\mathbb{C}$) of the Milnor fibre is not a useful
  invariant for distinguishing contact structures on links. We are
  nonetheless able to distinguish all these examples by studying a
  certain bigrading on symplectic cohomology, as we discuss in Section
  \ref{sct:intro_bigrading}. Note that Uebele's work (discussed in
  Example \ref{exm:ca1} above) shows that Conjecture
  \ref{conj:small_res} breaks down if we work over a field of
  characteristic \(\neq 0\), which gives an alternative way to
  distinguish contact structures on links.
\end{rmk}

\subsection{Families of inequivalent contact structures}
\label{sct:intro_bigrading}

We introduce the following notation for the contact structures on
the links of our singularities:

\begin{table}[htb]
  \label{tbl:ct_strs}
  \caption{Contact structures on links of our cDV singularities.}
  \begin{center}
    \begin{tabular}{llcc}
      \toprule
      &Singularity                                    & Link & Contact structure\\
      \midrule
      1.& \(x_1^2+x_2^2+x_3^{\enn+1}+x_4^{k(\enn+1)}\)          & \(\sharp_\enn(S^2\times S^3)\)  & \(\alpha_{\enn,k}\)\\
      2.& \(x_1^2+x_2^2+x_3x_4(x_3^{\enn-1}+x_4^{k(\enn-1)})\) (\(\enn\geq 2\))  & \(\sharp_\enn(S^2\times S^3)\)  & \(\beta_{\enn,k}\)\\
      3.& \(x_1^2+x_2^3+x_3^3+x_4^{6k}\)                & \(\sharp_4(S^2\times S^3)\)  & \(\delta_{4,k}\)\\
      4.& \(x_1^3+x_1x_2^{2k+1}+x_2x_3^2+x_4^2\)         & \(S^2\times S^3\)  & \(\lambda_{1,k}\)\\
      5.& \(x_1^2+x_2^3+x_3^4+x_4^{12k}\)               & \(\sharp_6(S^2\times S^3)\)  & \(\epsilon_{6,k}\)\\
      6.& \(x_1^2+x_2^3+x_3^5+x_4^{30k}\)               & \(\sharp_8(S^2\times S^3)\)  & \(\epsilon_{8,k}\)\\
      \bottomrule
    \end{tabular}
  \end{center}
\end{table}
  
\begin{rmk}
  Note that \(\alpha_{\enn,1}\cong\beta_{\enn,1}\): the two singularities
  are related by a change of variables.
\end{rmk}

Let \(\Xi_\enn\) denote the list of all contact structures on
\(\sharp_\enn(S^2\times S^3)\) from this table. For example,
\begin{align*}
  \Xi_1 &= (\alpha_{1,1}, \alpha_{1,2}, \ldots, \lambda_{1,1},
                  \lambda_{1,2}, \ldots)\\
  \Xi_4 &= (\alpha_{4,1}, \alpha_{4,2}, \ldots, \beta_{4,1},
                  \beta_{4,2}, \ldots, \delta_{4,1}, \delta_{4,2}, \ldots).
\end{align*}

\begin{thm}\label{thm:distinguish}
  For each \(\enn\), the contact structures in the list \(\Xi_\enn\)
  are pairwise nonisomorphic except for
  \(\alpha_{\enn,1}\cong\beta_{\enn,1}\).
\end{thm}
\begin{rmk}
  We remind the reader that all results about \(\lambda_{1,k}\) are
  conditional on a mirror symmetry statement.
\end{rmk}
\begin{rmk}
  What makes this an interesting theorem is that all of these links
  have the same positive symplectic cohomology over \(\mathbb{C}\). We
  equip \(\SH^*\) with a contact-invariant bigrading to distinguish
  these contact manifolds. This bigrading will be the weight
  decomposition of \(\bigoplus_{d<0}\SH^d\) under the action of the
  Lie algebra \(\SH^1\).
\end{rmk}
\begin{rmk}
  As explained in Example \ref{exm:ca1}, the fact that
  \(\alpha_{1,i}\not\cong\alpha_{1,j}\) if \(i\neq j\) was proved by
  Uebele \cite{Uebele} using positive symplectic cohomology with
  coefficients in \(\Z/2\) (rather than a bigrading).
\end{rmk}
\begin{rmk}
  If one focuses on cDV singularities which do not admit a small
  resolution, one finds very many more contact structures which can be
  distinguished by \(\SH^*_+\) already without using the
  bigrading. This is not so surprising: it is much easier for
  5-manifolds to be diffeomorphic than contactomorphic.
\end{rmk}

\noindent {\it Acknowledgments}.  J.~E. would like to thank Mark
McLean and Michael Wemyss for inspiration and for helpful
discussions. Y.~L. would like to thank Kazushi Ueda for
collaborations \cite{LU1},\cite{LU2} from which many ideas are
borrowed here. We would both like to thank the referees for
their thoughtful remarks and careful attention.

J.~E. was supported by EPSRC Grant EP/P02095X/2. Y.~L. was partially
funded by the Royal Society URF\textbackslash R\textbackslash 180024.

\section{Symplectic cohomology for invertible polynomials}
\label{sct:invertible}
\subsection{Symplectic cohomology}

Let $V$ be a Liouville manifold with $c_1(V)=0$. Associated to $V$ we
can define an invariant $\SH^*(V)$ called the symplectic cohomology of
$V$. Symplectic cohomology was introduced by Cieliebak, Floer, Hofer
\cite{H,FH,CFH} and Viterbo \cite{viterbo}. An excellent exposition
can be found in \cite{seidel}. More recent results can be learned from
\cite{CieliebakOancea}. See also \cite[Sec 2.1]{LPascaleff} for a fast
review of our sign and grading conventions. In particular, our
conventions are cohomological and the unit lives in degree zero!

Briefly, $\SH^*(V)$ is an algebra over the homology operad of framed
little discs over an arbitrary commutative ring $\mathbf{k}$ (in this
paper $\mathbf{k} =\mathbb{C})$. In particular, it has a (graded)
commutative product, a Gerstenhaber bracket $[\, ,\, ]$ (i.e a Lie
bracket of degree $-1$), and a Batalin-Vilkovisky operator $\Delta$
(i.e. a degree $-1$ operator whose Hochschild coboundary is the
bracket).

In general, symplectic cohomology is rather difficult to compute
explicitly. A fruitful approach to do such computations goes via the
open string A-model. Namely, we have an isomorphism
\[ \SH^*(V) \simeq \HH^*(\mathcal{W}(V)) \] where $\mathcal{W}(V)$ is
the wrapped Fukaya category of $V$. An early version of this result
based on Legendrian surgery is due to Bourgeois-Ekholm-Eliashberg
(\cite{BEE}, elaborated in \cite{EL}) which concerned Hochschild
homology; a definitive version based on duality appeared in
\cite[Theorem 1.1]{Ganatra} (see also the more recent
\cite{ghiggini}).

On the other hand, even if one achieved a good understanding of
$\mathcal{W}(V)$, in general, it is still a difficult algebraic
problem to compute Hochschild cohomology of $A_\infty$ categories.

In \cite{LU1}, \cite{LU2}, a method to compute symplectic cohomology
for certain Milnor fibres was given based on the homological mirror
symmetry conjecture for {\em invertible polynomials}.

\subsection{Invertible polynomials and mirror symmetry}
\label{sct:conjectures}
\begin{definition}
        To an \((n+1)\)-by-\((n+1)\) integer matrix \(A=(a_{ij})\) with nonzero
  determinant, we associate the polynomial
        \[\w(x_1,\ldots,x_{n+1})=\sum_{i=1}^{n+1}\prod_{j=1}^{n+1} x_j^{a_{ij}}.\] We
  write \(\check{\w}\) for the polynomial associated to \(A^T\) (the
  {\em Berglund-H\"{u}bsch mirror} to \(\w\), see \cite{BH}).
\end{definition}

An invertible polynomial is weighted homogeneous, that is there is a
uniquely determined weight system $(d_1,d_2,\ldots, d_{n+1};h)$ satisfying
$\gcd(d_1,d_2,\ldots, d_{n+1},h)=1$ for which
\[\w(\lambda^{d_1}x_1,\ldots,\lambda^{d_n}x_{n+1}) =
  \lambda^h\w(x_1,\ldots, x_{n+1})\] for all \(\lambda\in\G_m\). In this
paper, we are primarily concerned with the log Fano case, i.e. when
\[ h - \sum_{i=1}^{n+1} d_i =: d_0 < 0 \]

In fact, there is a finite extension \(\Gamma_\w\) of \(\G_m\) acting
on \(\A^{n+1}\) which preserves \(\w\), namely
\[\Gamma_\w := \left\{(t_0, t_1,\ldots,t_{n+1})\in\G_m^{n+2}\ :\
    \prod_{j=1}^{n+1} t_j^{a_{ij}} = t_0t_1\cdots t_{n+1},\
    i=1,\ldots,n+1\right\},\] acting on \(\A^{n+1}\) via
\((x_1,\ldots,x_{n+1})\mapsto
(t_1x_1,\ldots,t_{n+1}x_{n+1})\). This group also acts on
\(\A^{n+2}\) via
\((x_0, x_1,\ldots,x_{n+1})\mapsto (t_0x_0,
t_1x_1,\ldots,t_{n+1} x_{n+1})\), and this \(\Gamma_\w\)-action
preserves the polynomial
\[\w(x_1,\ldots,x_{n+1})+x_0\cdots x_{n+1}.\]
With this setup, we can formulate the following mirror symmetry
conjectures. A version of Conjecture \ref{conj:BH} appeared in
\cite{FutakiUedaSurvey} (see also {\cite[Conjecture 1.2]{LU1}} and
references therein), and Conjecture \ref{conj:LU} appeared in
\cite{LU1}.

\begin{conj}\label{conj:BH}
  There is a quasi-equivalence of idempotent complete
  \(A_\infty\)-categories
        \[\FF(\check{\w})\simeq \mf(\A^{n+1},\Gamma_\w,\w)\] between the
        Fukaya-Seidel category of a Morsification of
        \(\check{\w}\) and the dg-category of
        \(\Gamma_\w\)-equivariant matrix factorisations of
        \(\w\). Moreover, there exists a full exceptional collection
        \(\Delta_1,\ldots,\Delta_K\) of vanishing thimbles for
        the Morsification of \(\check{\w}\) such that the
        \(A_\infty\)-algebra
        \(\AA :=
        \OP{end}_{\FF(\check{\w})}\left(\bigoplus_i\Delta_i\right)\)
        has its cohomology \(A:=H(\AA)\) supported in degree
        zero. In particular, this entails that (a) \(\AA\) is
        quasi-isomorphic to \(A\) and (b) both
        \(\FF(\check{\w})\) and \(\mf(\A^{n+1},\Gamma_\w,\w)\)
        are quasi-equivalent to \(\OP{perf}(A)\).
      \end{conj}

\begin{conj}[{\cite[Conjecture 1.4]{LU1}}]\label{conj:LU}
 There is a quasi-equivalence of idempotent complete $A_\infty$ categories 
        \[ \mathcal{W}(\check{\w}^{-1}(1)) \simeq \mf(\A^{n+2}, \Gamma_\w,
        \w+x_0x_1\cdots x_{n+1}) \] between the wrapped Fukaya category of the
 Milnor fibre \(\check{\w}^{-1}(1)\) and the dg-category of
 \(\Gamma_\w\)-equivariant matrix factorisations of
        \(\w+x_0\cdots x_{n+1}\).
\end{conj}

These conjectures are established in the following situations:
\begin{itemize}
\item If the matrix \(A\) is diagonal (so \(\w\) defines a
  Brieskorn-Pham singularity) then Conjecture \ref{conj:BH} was proved
  by Futaki and Ueda \cite{FutakiUeda}. More generally, if the matrix
  \(A\) is block diagonal and its blocks are either \(1\)-by-\(1\) or
  \(2\)-by-\(2\) equal to
  \(\begin{pmatrix}2 & 1 \\ 0 & k\end{pmatrix}\) (so that \(\w\) is a
  Sebastiani-Thom sum of ADE polynomials of type \(A\) or \(D\)),
  Conjecture \ref{conj:BH} was proved by Futaki and Ueda
                \cite{FutakiUedaD}. Polishchuk and Varolgunes \cite{PolishchukVarolgunes} make significant progress towards establishing Conjecture \ref{conj:BH} in the chain case which includes the Laufer flop (Case 4 in the Table of Theorem
  \ref{thm:summary}). 
\item If \(n=1\), Conjecture \ref{conj:BH} was proved by Habermann and
  Smith \cite{HabermannSmith}. In fact, this means Conjecture
  \ref{conj:BH} holds for any invertible polynomial \(\w\) of the form
                \(\w(x_1,\ldots,x_{n+1})=x_1^2+\cdots+x_{n-1}^2+f(x_{n},x_{n+1})\). This
  is because stabilising \(\w\) and \(\check{\w}\) by adding quadratic
  terms in extra variables changes neither the Fukaya-Seidel nor the
  matrix factorisation category.
\item In \cite{LU1}, various cases of Conjecture \ref{conj:LU} were
  verified. The sequel paper \cite{LU2} focused on the log Fano case
  and established Conjecture \ref{conj:LU} for the Milnor fibres of
                simple singularities. The \( n=1 \) case of Conjecture \ref{conj:LU} was proved by Habermann \cite{Habermann}. Conjecture \ref{conj:LU} was proved in full
  generality by Gammage \cite{Gammage} in the $\mathbb{Z}/2$-graded
  case using a microlocal sheaf category version of wrapped Fukaya
  categories. For our purposes, we will need to work with
  $\mathbb{Z}$-graded categories; a careful chase of
  $\mathbb{Z}$-gradings in \cite{Gammage} might allow us to assume
  Conjecture \ref{conj:LU} in all cases.
\end{itemize}

\begin{rmk}
  The main theorem statements from Futaki-Ueda and Habermann-Smith do
  not mention the formality of \(\AA\), but in either case the authors construct a full exceptional collection whose cohomology is supported in degree zero, hence formality follows for degree reasons
\end{rmk}

\begin{rmk}
  The examples \(\check{\w}\) from Theorem \ref{thm:summary} all fall
  into one of these cases except for \(\check{\w} = x_1^3 + x_1x_2^{2k+1} +
  x_2x_3^2 + x_4^2\). In this case, our results are conditional on one
  of the two Conjectures \ref{conj:BH} or \ref{conj:LU} holding.
\end{rmk}

We now explain how knowing one or other of these conjectures can help
one to calculate symplectic cohomology.

\subsection{Using mirror symmetry to compute symplectic cohomology}
\label{sct:mirror}

Pick a Morsification of \(\check{\w}\). Let \(\FF(\check{\w})\) denote
the Fukaya-Seidel category of the Morsification, let
\(V:=\check{\w}^{-1}(1)\) denote the Milnor fibre, and let \(\WW(V)\)
(respectively \(\FF(V)\)) denote the wrapped (respectively compact)
Fukaya category of \(V\). Choose a collection of vanishing paths for
the Morsification and let \(\Delta_1,\ldots,\Delta_K\) (respectively
\(S_1,\ldots,S_K\)) be the corresponding vanishing thimbles
(respectively vanishing cycles). Let
\(\AA = \OP{end}_{\FF(\check{\w})}\left(\bigoplus_i \Delta_i\right)\)
and \(\BB = \OP{end}_{\FF(V)}\left(\bigoplus_i S_i\right)\). Let
\(A=H(\AA)\) and \(B=H(\BB)\) denote the cohomology algebras of
\(\AA\) and \(\BB\) (considered as \(A_\infty\)-algebras with zero
higher products).

\begin{thm}\label{thm:conditional}
  Assume that \(\HH^2(\mf(\A^{n+2},\Gamma_{\w},\w))=0\), that
  \(d_0\neq 0\), and either Conjecture \ref{conj:BH} or Conjecture
  \ref{conj:LU}
  holds. Then \begin{equation}\label{eq:sh_mf}\SH^*(V)\cong
    \HH^*(\mf(\A^{n+2},\Gamma_{\w},\w))\end{equation} as Gerstenhaber
  algebras.
\end{thm}

In the next section, we give a formula to compute
\(\HH^*(\mf(\A^{n+2},\Gamma_\w,\w))\).

\begin{proof}[Proof that Conjecture \ref{conj:BH} implies Equation
  \eqref{eq:sh_mf}]
  If \(d_0\neq 0\) then {\cite[Theorem 6.2]{LU1}} implies that the
  inclusion of categories \(\BB\to\WW(V)\) induces an isomorphism on
  Hochschild cohomology. Since this map comes from a functor, it is a
  morphism of Gerstenhaber algebras. Ganatra {\cite[Theorem
    1.1]{Ganatra}} shows that \(\SH^*(V)\cong\HH^*(\WW(V))\) as
  Gerstenhaber algebras. Therefore, we need to
  show \begin{equation}\label{eq:goal}\HH^*(\BB)\cong
    \HH^*(\mf(\A^{n+2},\Gamma_{\w},\w)).\end{equation} As a first
  step, we calculate \(\HH^*(B)\), where \(B= H(\BB)\) is the
  cohomology algebra of \(\BB\).

  \begin{lem}\label{lma:hhs}
    We have \(\HH^*(B)\cong\HH^*(\mf(\A^{n+2},\Gamma_{\w},\w))\)
    as Gerstenhaber algebras. 
  \end{lem}
  \begin{proof}
    We continue to write \(\AA\) for the endomorphism
    \(A_\infty\)-algebra of the vanishing thimbles and \(A\) for
    its cohomology. Recall that the {\em trivial extension
      algebra} \(T_n(A)\) is defined
    to be \(A\oplus A^\vee[-n]\) with the product
    \((a,b)(a',b')=(aa',ab'+a'b)\). For any Lefschetz fibration
    with \((n+1)\)-(complex-)dimensional total space with $n>0$,
    the Floer cohomology algebra \(B= H(\BB)\) of the vanishing
    cycles is an extension of \(A^\vee[-n]\) by \(A\), where
    \(A\) is the directed Fukaya-Seidel Floer cohomology algebra
    for the vanishing thimbles {\cite[Equation 4.1 and
      Proposition 5.1]{SeidelSuspending}}. If \(A\) is supported
    in degree zero (as asserted by Conjecture 2.2) then this is
    the trivial extension \(T_n(A)\): the products in \(B\)
    which are not determined by the \(A\)-module structure of
    \(A^\vee[-n]\) vanish for degree reasons. To prove the
    lemma, it therefore suffices to show that
    \(\HH^*(T_n(A))\simeq \HH^*(\mf(\A^{n+2},\Gamma_{\w},\w))\).

    Let \(\k\) be the semisimple ring
    \(\bigoplus_{i=1}^K\C e_i\) where \(e_i\in A\) is the
    identity element of \(HF(\Delta_i,\Delta_i)\). The
    projection \(T_n(A)\to A\to\k\) makes \(\k\) into an
    \(T_n(A)\)-module (augmentation). Keller {\cite[Section
      4.1]{Keller}} defines a Koszul-dual algebra called the
    {\em \(n\)-Calabi-Yau completion}
    \(\Pi_n(A)\cong\OP{RHom}_{T_n(A)}(\k, \k)\). This is Koszul-dual
    in the sense that \(\k\) is a \((T_n(A),\Pi_n(A))\)-bimodule
    and \(T_n(A)\cong\OP{RHom}_{\Pi_n(A)}(\k, \k)\). Koszul duality
    ensures that we can apply {\cite[Theorem in Section
      3.2]{Keller2}} to deduce that the Hochschild cohomologies
    \(\HH^*(T_n(A))\) and \(\HH^*(\Pi_n(A))\) are isomorphic as
    Gerstenhaber algebras. For any algebra \(C\) (more generally
    \(A_\infty\)-algebra),
    \(\HH^*(C)\cong \HH^*(\OP{perf}(C))\), so
    \[\HH^*(B)\cong \HH^*(T_n(A))\cong \HH^*(\Pi_n(A))\cong
      \HH^*(\OP{perf}(\Pi_n(A))),\] and it suffices to prove
    that
    \(\HH^*(\OP{perf}(\Pi_n(A))) =
    \HH^*(\mf(\A^{n+2},\Gamma_{\w},\w))\). In fact, we will show
    a stronger result: that
    \[\OP{perf}(\Pi_n(A))\simeq
      \mf(\A^{n+2},\Gamma_{\w},\w).\] To see this stronger
    result, recall that Keller's construction of
    \(\Pi_n(\cdot)\) works more generally when the input is a
    dg-algebra or category, and satisfies
    {\cite[Eq. (2.2)]{LU2}}
    \[\OP{perf}(\Pi_n(A))\simeq\Pi_n(\OP{perf}(A)).\] It was
    shown in {\cite{LU2}} (Eq. (1.7) for the statement and
    Section 4 for the proof) that
    \begin{equation}\label{eq:mf_ncy}\mf(\A^{n+2},\Gamma_{\w},\w)\simeq
      \Pi_{n}(\mf(\A^{n+1},\Gamma_{\w},\w))\end{equation}
    and Conjecture \ref{conj:BH} is the assumption that
    \[\OP{perf}(A)\simeq \mf(\A^{n+1},\Gamma_{\w},\w)\]
    so    
    \[\OP{perf}(\Pi_n(A))\simeq \Pi_n(\OP{perf}(A))\simeq
      \Pi_n(\mf(\AA^{n+1},\Gamma_{\w},\w))\simeq
      \mf(\AA^{n+2},\Gamma_{\w},\w),\] as required.
  \end{proof}

  \begin{lem}\label{lma:formal}
    The \(A_\infty\)-algebra \(\BB\) is quasi-isomorphic to its
    cohomology algebra \(B\).
  \end{lem}
  \begin{proof}
    By Lemma \ref{lma:hhs},
    \(\HH^2(B)\cong\HH^2(\mf(\A^{n+2},\Gamma_{\w},\w))\), which
    vanishes by assumption, so \(\BB\) is intrinsically formal, and
    hence quasi-isomorphic to \(B=H(\BB)\).
  \end{proof}

  Together, these two lemmas show that
  \(\HH^*(\BB)\cong\HH^*(\mf(\A^{n+2},\Gamma_{\w},\w))\) as
  Gerstenhaber algebras, establishing Equation \eqref{eq:goal}, so
  Equation \eqref{eq:sh_mf} follows.
\end{proof}

\begin{proof}[Proof that Conjecture \ref{conj:LU} implies Equation
  \eqref{eq:sh_mf}]
  We will show in Theorem \ref{thm:arnold} below that if
  \(\HH^2(\mf(\A^{n+2},\Gamma_{\w},\w))=0\) then we can make a
  \(\Gamma_{\w}\)-equivariant formal change of coordinates along the
        critical locus of \(\w+x_0\cdots x_{n+1}\) such that the pullback of
        \(\w+x_0\cdots x_{n+1}\) in these new coordinates equals \(\w\). If we
  can make such a formal change of coordinates, it follows from
  {\cite[Theorem 2.10]{Orlov}} that
        \(\mf(\A^{n+2},\Gamma_\w, \w+x_0\cdots x_{n+1})\) is quasi-equivalent to
  \(\mf(\A^{n+2},\Gamma_{\w},\w)\), so Conjecture \ref{conj:LU}
  implies
  \[\HH^*(\WW(V))\cong\HH^*(\mf(\A^{n+2},\Gamma_{\w},\w))\] as
  Gerstenhaber algebras. By {\cite[Theorem 1.1]{Ganatra}},
  \(\HH^*(\WW(V))\cong \SH^*(V)\) as Gerstenhaber algebras, so
  Equation \eqref{eq:sh_mf} follows.
\end{proof}

\subsection{Calculating \(\HH^*(\mf(\A^{n+2},\Gamma_\w,\w))\)}

There is a formula for \(\HH^*(\mf(\A^{n+2},\Gamma_\w,\w))\)
which expresses it as a sum of $\Gamma_\w$-invariant pieces of
twisted Koszul cohomologies; this formula appeared in
{\cite[Theorem 1.2]{BFK}}, where its context and history are
discussed. It is also explained and used in {\cite[Theorem
  3.1]{LU1}} and {\cite[Section 5.1]{LU2}}. We now briefly
describe how to perform calculations in practice with this
formula; Theorem \ref{thm:hh_formula} below summarises the
answer and its proof explains how our notation fits with the
notation from \cite{LU2}. We will use the notation from this
section in our calculations in Section \ref{sct:compendium}.

\begin{definition}
  Define the character
        \[\chi\colon\Gamma_\w\to \G_m,\quad \chi(t_0,\ldots,t_{n+1})=t_0\cdots
    t_{n+1}.\] Its kernel \(\ker\chi\) is the finite group
        \[\ker\chi=\left\{(t_0,\ldots,t_{n+1})\in\G_m^{n+2} :\ \prod_{j=1}^{n+1} t_j^{a_{ij}}=1,\
        t_0=t_1^{-1}\cdots t_{n+1}^{-1}\right\}.\]
\end{definition}

\begin{definition}
  Given an element \(\gamma\in\ker\chi\), let
        \[\{1,\ldots,{n+1}\} = \{i_1,\ldots,i_k\}\cup\{j_1,\ldots,j_{n+1-k}\}\] be
  the partition for which each \(x_{i_m}\) is fixed under the action
  of \(\gamma\) and each \(x_{j_m}\) is not fixed under the action of
  \(\gamma\). Let \(J_\gamma\) be a monomial basis for the Jacobian
  ring of \(\w|_{x_{j_1}=\cdots=x_{j_{n+1-k}}=0}\).
\end{definition}

\begin{definition}
  The set \(M_\gamma\) of {\em \(\gamma\)-monomials} is the union
  \(M_\gamma = A_\gamma \cup B_\gamma \cup C_\gamma\) where
  \begin{align*}
    A_\gamma &= \begin{cases}
                 \{x_0^\beta p x_{j_1}^{\vee} \cdots x_{j_{n+1-k}}^{\vee}\ :
                 \ p\in J_\gamma,\ \beta=0,1,2,\ldots\} & \mbox{ if }x_0\mbox{ is fixed by }\gamma \\
                 \emptyset&\mbox{ otherwise.}
               \end{cases}\\
    B_\gamma &= \begin{cases} \{x_0^\beta p x_0^{\vee} x_{j_1}^{\vee}
                  \cdots x_{j_{n+1-k}}^{\vee}\ :
                  \ p\in J_\gamma,\ \beta=0,1,2,\ldots\} & \mbox{ if }x_0\mbox{ is fixed by }\gamma \\
                  \emptyset&\mbox{ otherwise.}
               \end{cases}\\
    C_\gamma &= \begin{cases}
                 \emptyset & \mbox{ if }x_0\mbox{ is fixed by }\gamma \\
                 \{p x_0^{\vee} x_{j_1}^{\vee}\cdots x_{j_{n+1-k}}^{\vee}\ :
                 \ p\in J_\gamma\}&\mbox{ otherwise.}
               \end{cases}
  \end{align*}
\end{definition}

\begin{definition}
  Let \(\zeta\colon\Gamma_{\w}\to\G_m\) be a character of
  \(\Gamma_{\w}\). We say that a polynomial or formal power series
  \(p(x_0,\ldots,x_{n+1})\) is {\em \(\zeta\)-isotypical} if
  \(p(g\bm{x})=\zeta(g)p(\bm{x})\) for all \(g\in \Gamma_{\w}\). Note
  that every monomial \(m\) determines a character \(\xi(m)\) such
  that \(m\) is \(\xi(m)\)-isotypical. The space of formal power
  series \(K:=\C[\![x_0,\ldots,x_{n+1}]\!]\) is therefore the
  completed direct sum of its \(\zeta\)-isotypical summands
  \[K=\widehat{\bigoplus}_{\zeta\in\hat{\Gamma}_{\w}}K_\zeta,\quad K_\zeta
    = \left\{p\in K\ :\ p(g\bm{x}) = \zeta(g)p(\bm{x})\ \forall
      g\in\Gamma_{\w}\right\}.\]
\end{definition}

\begin{definition}
  Given a \(\gamma\)-monomial \(m\), we write \(b_k\) for the total
  exponent of \(x_k\) in \(m\), where \(x_k^{\vee}\) contributes
  \(-1\) to \(b_k\). The character \(\xi(m)\) is determined by these
  exponents:
  \[\xi(m)(t_0,\ldots,t_{n+1})=t_0^{b_0}\cdots t_{n+1}^{b_{n+1}}.\]
\end{definition}

We now assume the following. For each \(\gamma\in\Gamma_{\w}\)
let \(\w_\gamma\) (respectively \(\w'_\gamma\)) denote the
restriction of the polynomial \(\w\) to the subspace where the
unfixed variables \(x_{j_1},\cdots,x_{j_{n-k+1}}\) (respectively
\(x_0,x_{j_1},\cdots,x_{j_{n-k+1}}\)) vanish. We assume that
\(\w'_\gamma\) has an isolated singularity at the origin for all
\(\gamma\in\Gamma\), which is the case for all our examples.

\begin{thm}\label{thm:hh_formula}
  Under this assumption, the Hochschild cohomology
  \(\HH^*(\mf(\A^{n+2},\Gamma_{\w},\w))\) is a direct sum of
  1-dimensional contributions, one from each pair \((\gamma,m)\)
  with \(m\in M_\gamma\) such that \(\xi(m)=\chi^{\otimes u}\)
  for some \(u\in \Z\). In this case, \((\gamma,m)\) contributes
  to
  \begin{align*}
    \HH^{2u+n-k+1}&\mbox{ if }m\in A_{\gamma},\\
    \HH^{2u+n-k+2}&\mbox{ if }m\in B_{\gamma},\\
    \HH^{2u+n-k+2}&\mbox{ if }m\in C_{\gamma},\\
  \end{align*}
  where \(k\) is the number of variables amongst
  \(\{x_1,\ldots,x_{n+1}\}\) fixed by \(\gamma\).
\end{thm}
\begin{proof}
  This is just a repackaging of {\cite[Theorem 1.2]{BFK}}, based
  on the exposition in {\cite[Section 5.1]{LU2}}. We briefly
  explain how to translate between our notation and the notation
  of \cite{LU2}. For each \(\gamma\) there are three kinds of
  contribution to Hochschild cohomology, enumerated by
  {\cite[Equations 5.5-5.7]{LU2}}:
  \begin{itemize}
  \item If \(x_0\) is not fixed by \(\gamma\) then the Hochschild
    cohomology picks up a contribution given by {\cite[Equation
      5.5]{LU2}}:
    \[\left(\OP{Jac}_{\w_\gamma}\otimes\Lambda^{\dim
          N_\gamma}N_\gamma^\vee\right)_{\chi^{\otimes u}}\]
    where \(N_\gamma\) is the vector space spanned by the
    non-fixed variables \(x_0,x_{j_1},\ldots,x_{j_{n-k+1}}\),
    \(\OP{Jac}\) denotes the Jacobian ring, and
    \(\chi^{\otimes u}\) means taking the isotypical part. Our
    \(\gamma\)-monomials from \(C_\gamma\) form an explicit
    basis of this space: \(J_\gamma\) is a basis for the
    Jacobian \(\OP{Jac}_{\w_\gamma}\) and
    \(x_0^\vee\otimes x_{j_1}^\vee\otimes\cdots\otimes
    x_{j_{n-k+1}}^\vee\) is a generator of
    \(\Lambda^{\dim N_\gamma}N_\gamma^\vee\); the
    \(\chi^{\otimes u}\) subscript is precisely telling us to
    restrict attention to \(\gamma\)-monomials with
    \(\xi(m)=\chi^{\otimes u}\). This contributes to
    \(\HH^{2u+\dim N_\gamma}=\HH^{2u+n-k+2}\).
  \item If \(x_0\) is fixed by \(\gamma\) then there are
    contributions {\cite[Equations 5.6 and 5.7]{LU2}}:
    \[\left(\OP{Jac}_{\w'_\gamma}\otimes\C[x_0]\otimes\Lambda^{\dim
          N_\gamma}N_\gamma^\vee\right)_{\chi^{\otimes
          u}},\qquad \left(\C
        x_0^\vee\otimes\OP{Jac}_{\w'_\gamma}\otimes\C[x_0]\otimes\Lambda^{\dim
          N_\gamma}N_\gamma^\vee\right)_{\chi^{\otimes u}}\] to
    \(\HH^{2u+\dim N_\gamma}=\HH^{2u+n-k+1}\) and
    \(\HH^{2u+\dim N_\gamma+1}=\HH^{2u+n-k+2}\)
    respectively. Our \(\gamma\)-monomials of type \(A_\gamma\)
    and \(B_\gamma\) give bases for these vector spaces.
  \end{itemize}
\end{proof}

\subsection{Formal change of coordinates}

In this section we prove the last remaining ingredient (Theorem
\ref{thm:arnold} below) that was used in Section
\ref{sct:mirror} above (in the proof that Conjecture
\ref{conj:LU} implies Equation \eqref{eq:sh_mf}).

Recall that there exist weights \(d_0,d_1,\ldots,d_{n+1}\) such that if we
give \(x_i\) weight \(d_i\) then both \(\w(x_1,\ldots,x_{n+1})\) and
\(x_0\cdots x_{n+1}\) are quasihomogeneous of degree \(h=\sum_{i=0}^{n+1} d_i\)
and \(\chi\)-isotypical. Let \(|\cdot|_0\) be the \(x_0\)-valuation on
the space \(K=\C[\![x_0,\ldots,x_{n+1}]\!]\) of formal power series,
i.e. \(|p|_0=k\) if \(x_0^k\) divides \(p\) but \(x_0^{k+1}\) does
not.

\begin{thm}\label{thm:arnold}
        Suppose that \(p_0(x_0,\ldots,x_{n+1})\) is a \(\chi\)-isotypical formal
  power series which is quasihomogeneous of degree \(h\) and
  \(|p_0|_0>0\). If \(\HH^2(\mf(\A^{n+2},\Gamma_{\w},\w))=0\) then there is
  a formal change of variables
        \(\bm{z}=(x_0,z_1(\bm{x}),\ldots,z_{n+1}(\bm{x}))\) such that
  \(\w(\bm{z})=\w(\bm{x})+p_0(\bm{x})\).
\end{thm}
\begin{rmk}
  In particular, the theorem applies when
        \(p_0(\bm{x}) = x_0\cdots x_{n+1}\). To prove Theorem \ref{thm:arnold},
  we first establish a sequence of lemmas.
\end{rmk}

\begin{lem}
  Suppose that \(\HH^2(\mf(\A^{n+2},\Gamma_{\w},\w))=0\). Then the image
  of \(K_\chi\) in the Jacobian ring is trivial.
\end{lem}
\begin{proof}
  If \(m\) is a \(\chi\)-isotypical monomial which is nontrivial in
  the Jacobian ring then we can use it as part of our monomial basis
  \(J_\gamma\) for \(\gamma=id\). It will then contribute as a type
  \(A\) \(id\)-monomial to \(\HH^2(\mf(\A^{n+2},\Gamma_{\w},\w))\). Thus if
  \(\HH^2(\mf(\A^{n+2},\Gamma_{\w},\w))=0\), we deduce that any monomial
  \(m\in K_\chi\) is trivial in the Jacobian ring, and hence the image
  of \(K_\chi\) in the Jacobian ring is zero.
\end{proof}

\begin{lem}
  If \(p\in K_\chi\) is trivial in the Jacobian ring then
        \(p=\sum_{i=1}^{n+1} v_i\frac{\partial \w}{\partial x_i}\) for some
        \(v_1,\ldots,v_{n+1} \in K\) where \(v_i\in K_{t_i}\). Here, \(t_i\)
  denotes the character of \(\Gamma_{\w}\) which projects to \(t_i\).
\end{lem}
\begin{proof}
  Consider the map \(\partial\colon K^{n+1} \to K\) defined by
  \(\partial (v_1, \ldots, v_{n+1})=\sum_{i=1}^{n+1}
  v_i\frac{\partial\w}{\partial x_i}\). The cokernel of \(\partial\)
  is the Jacobian ring. Because \(\w\in K_\chi\), we have
  \(\partial\w/\partial x_i\in K_{\chi\otimes t_i^{-1}}\) for all
  \(i\), so \(\partial(v_1,\ldots,v_{n+1} )\in K_\chi\) if and only if
  \(v_i\in K_{t_i}\) for all \(i=1,\ldots,n+1\).
\end{proof}
  
If \(\bm{v}\in K_{t_1}\oplus\cdots\oplus K_{t_{n+1}}\) then we call
\(\bm{v}\) a \(\Gamma_{\w}\)-equivariant vector field because the
components \(v_i\) of \(\bm{v}\) transform under \(\Gamma_{\w}\) like the
coordinates \(x_i\). We have now seen that, under the hypotheses of
Theorem \ref{thm:arnold}, \(p_0=\partial \bm{v}\) for a
\(\Gamma_{\w}\)-equivariant vector field \(\bm{v}\).

\begin{lem}\label{lma:iteration}
  In the setting of Theorem \ref{thm:arnold}, there exists a formal
  change of variables \(\bm{y}\) such that
  \[p_1(\bm{y}):=\w(\bm{x})+p_0(\bm{x}) - \w(\bm{y})\] is
  \(\chi\)-isotypical and satisfies \(|p_1|_0>|p_0|_0\).
\end{lem}
\begin{proof}
  This is a small modification of {\cite[Section 12.6]{Arnold}}. We
  know that \(p_0=\partial\bm{v}\) for a \(\Gamma_{\w}\)-equivariant
  vector field \(\bm{v}\). We define \(\bm{y}\) implicitly by
  \(y_0 = x_0,\quad x_i=y_i-v_i(\bm{y})\). Since \(\bm{v}\) is
  \(\Gamma_{\w}\)-equivariant, this formal change of coordinates is
  \(\Gamma_{\w}\)-equivariant. As in the proof\footnote{In
    \cite{Arnold}, they have no variable \(x_0\) and filter by the
    weighted degree of the perturbation rather than the
    \(x_0\)-valuation. Since the perturbation terms are
    quasihomogeneous of degree \(h\), the weighted degree of the
    perturbation term with respect to \(x_1,\ldots,x_{n+1}\) is
    proportional to the \(x_0\)-valuation, so our strategy is
    equivalent.} of {\cite[Section 12.6]{Arnold}}, we find that
  \(p_1(\bm{y}):=\w(\bm{x})+p_0(\bm{x})-\w(\bm{y})\) has
  \(|p_1|_0>|p_0|_0\). Moreover, since \(\bm{v}\) is
  \(\Gamma_{\w}\)-equivariant, \(p_1\) is \(\chi\)-isotypical.
\end{proof}

\begin{proof}[Proof of Theorem \ref{thm:arnold}]
  We can apply Lemma \ref{lma:iteration} iteratively and compose the
  formal diffeomorphisms we get at each stage. Composition makes sense
  because \(|\bm{v}|_0>0\), so the \(x_0^k\)-term in the composition
  of formal diffeomorphisms only involves summing finitely many
  terms. In this way, we construct a sequence of perturbation terms
  \(p_1,p_2,\ldots\) with \(|p_1|_0<|p_2|_0<\cdots\). In the limit, we
  obtain a formal change of coordinates \(\bm{z}\) with perturbation
  term \(p_\infty(\bm{z}):=\w(\bm{x}) + p_0(\bm{x}) - \w(\bm{z})\)
  satisfying \(|p_\infty|_0=\infty\). Therefore \(p_\infty = 0\) and
  we have proved the theorem.
\end{proof}

\section{Compendium of examples}
\label{sct:compendium}

In this section, we calculate
\(\HH^* := \HH^*(\mf(\A^{n+2},\Gamma_\w,\w))\) for the invertible
polynomials \(\w\) which are mirror-dual to the polynomials in Theorem
\ref{thm:summary}. We now summarise how this leads to a proof of that
theorem.

\begin{proof}[Proof of Theorem \ref{thm:summary}]
  These examples are log Fano, so \(d_0<0\), and, in all cases, we
  will see that \(\HH^2(\mf(\A^{n+2},\Gamma_\w,\w)) = 0\). In Cases
  1--3 and 5--6 of Theorem \ref{thm:summary}, Conjecture \ref{conj:BH}
  holds, so that Theorem \ref{thm:conditional} applies. As a
  consequence, we can conclude Theorem \ref{thm:summary}
  unconditionally in these cases. Case 4 holds conditionally on
  Conjecture \ref{conj:BH} or \ref{conj:LU}.
\end{proof}

\subsection{Brieskorn-Pham}

A Brieskorn-Pham singularity is an isolated hypersurface singularity
given by the vanishing of the polynomial
\[\w(x_1, \ldots, x_{n+1}) = x_1^{a_1} + \cdots + x_{n+1}^{a_{n+1}}\] for a collection of integers
\(a_i\geq 2\). This is an invertible polynomial with
\(\check{\w}=\w\).

Let \(\bm{\mu}_k\) denote the cyclic group of \(k\)th roots of unity,
and let
\[ \lle =\lcm(a_1,\ldots,a_{n+1}),\qquad
  \ohr=1-\sum_{i=1}^{n+1}\frac{1}{a_i}.\] We have a surjective
\(l\)-to-\(1\) homomorphism
\begin{align*}
  \mho&\colon\bm{\mu}_{a_1}\times\cdots\times\bm{\mu}_{a_{n+1}}\times\G_m\to\Gamma_{\w},\\
  \mho&\left(\mu_1,\ldots,\mu_{n+1},\tau\right)=
        \left(\tau^{\lle\ohr}\mu_1^{-1}\cdots\mu_{n+1}^{-1},
        \tau^{\lle/a_1}\mu_1, \ldots, \tau^{\lle/a_{n+1}}\mu_{n+1}\right)
\end{align*}

\begin{rmk}
  Under \(\mho\), the subgroup
        \(\bm{\mu}_{a_1}\times\cdots\times\bm{\mu}_{a_{n+1}}\) maps
  isomorphically onto \(\ker\chi\); we will use this identification to
        write elements of \(\ker\chi\) as \((n+1)\)-tuples of roots of unity.
\end{rmk}

Fix an element \(\gamma\in\ker\chi\). Restricting \(\w\) to the fixed
variables \(x_{i_1},\ldots,x_{i_k}\) we get
\(\sum_{m=1}^k x_{i_m}^{a_{i_m}}\), and we pick the monomial basis
\[J_{\gamma} = \left\{x_{i_1}^{b_{i_1}}\cdots x_{i_k}^{b_{i_k}}\ :\
  0\leq b_{i_m}\leq a_{i_m}-2\mbox{ for }m=1,2,\ldots k\right\}\] for
its Jacobian ring.

\begin{lem}
  Let \(m\) be a \(\gamma\)-monomial with total exponents
        \(b_0,\ldots,b_{n+1}\) and suppose that \((\gamma, m)\) contributes to
  Hochschild cohomology. Then \(b_i = b_0\mod a_i\) for
  \(i=1,\ldots,n+1\) and
        \(\xi(m)=\chi^{\otimes (b_0-\sum_{i=1}^{n+1} m_i)}\), where the integers
  \(m_i\) are determined by \(b_0 = b_i + m_ia_i\).
\end{lem}
\begin{proof}
  The \(\gamma\)-monomial \(m\) with total exponents
        \(b_0,\ldots,b_{n+1}\) has character\footnote{The characters of
    \(\Gamma_{\w}\) induce characters of
    \(\bm{\mu}_{a_1}\times\cdots\times\bm{\mu}_{a_{n+1}}\times\G_m\) by
    precomposing with \(\mho\) and we will often write characters of
    \(\Gamma_{\w}\) by giving a character of the bigger group which factors
    through \(\mho\).}
  \[\xi(m)(\mu_1,\ldots,\mu_{n+1},\tau) =
    \tau^{b_0\lle\ohr+\sum_{i=1}^{n+1}b_i\lle/a_i} \mu_1^{b_1-b_0}
    \cdots \mu_{n+1}^{b_{n+1}-b_0}\] This coincides with a power of
  \(\chi\) if and only if \(b_0=b_i\mod a_i\) for
  \(i=1,\ldots,n+1\). More precisely, if \(b_0 = b_i + m_ia_i\) for
  integers \(m_1,\ldots,m_n\) then
        \(\xi(m)=\chi^{\otimes (b_0-\sum_{i=1}^{n+1} m_i)}\).
\end{proof}

\begin{rmk}
  In fact, if \(b_0\geq 0\), then \(b_0\) uniquely determines
  monomials \(m_A(b_0)\) and \(m_B(b_0)\) of types \(A\) and \(B\)
  respectively which have total exponents \(b_i=b_0\mod a_i\). Namely,
  we multiply together factors of \(x_i^{b_i}\), \(i=1,\ldots,n+1\),
  where \(x_i^{-1}\) means \(x_i^\vee\). To obtain \(m_A(b_0)\) we
  include a factor of \(x_0^{b_0}\); to obtain \(m_B(b_0)\) we include
  a factor of \(x_0^{b_0+1}x_0^\vee\). Similarly, \(b_0=-1\)
  determines unique monomials
        \(m_B(-1)=x_0^\vee\cdots x_{n+1}^\vee\) of type \(B\) and
        \(m_C(-1)=x_0^\vee\cdots x_{n+1}^\vee\) of type \(C\).
\end{rmk}

\begin{rmk}
  By the Sun Zi remainder theorem, given any collection of total
        exponents \(0\leq b_1 < a_1,\ \ldots,\ 0\leq b_{n+1} < a_{n+1}\), we can
  solve this system of congruences for \(b_0\) uniquely modulo
  \(\lle\) if and only if \(b_i=b_j\mod\gcd(a_i,a_j)\) for all
        \(i,j\in\{1,\ldots,{n+1}\}\).
\end{rmk}

Our approach to calculating \(\HH^*\) will therefore be to consider each
possible value \(b_0\) and find the number of elements
\(\gamma\in\ker\chi\) such that \((\gamma, m_A(b_0))\),
\((\gamma, m_B(b_0))\) or \((\gamma, m_C(b_0))\) is a contributing
\(\gamma\)-monomial.

The contributions from \(b_0=-1\) are easy to calculate.

\begin{lem}\label{lma:-1}
  The contributions from monomials with total exponent \(b_0=-1\) come
  from
        \[(\gamma, x_0^\vee\cdots x_{n+1}^\vee)\in\HH^{n}\] for
  all
  \(\gamma\in \left(\bm{\mu}_{a_1}\setminus\{1\}\right) \times \cdots
        \times \left(\bm{\mu}_{a_{n+1}}\setminus\{1\}\right)\).
\end{lem}
\begin{proof}
        We have \(m_B(-1)=m_C(-1)=x_0^\vee\cdots x_{n+1}^\vee\). This
  contributes as a \(\gamma\)-monomial if and only if either:
  \begin{itemize}
          \item \(\gamma\) leaves all variables \(x_0,\ldots,x_{n+1}\) unfixed. In
    this case we get a type \(C\) contribution from
    \((\gamma, m_C(-1))\).
  \item \(\gamma\) fixes \(x_0\) and does not fix any other
    variable. In this case we get a type \(B\) contribution from
    \((\gamma, m_B(-1))\).\qedhere
  \end{itemize}
\end{proof}

\begin{rmk}
        In fact, in our examples (but not in general), these will be the only contributions to \(\HH^n\), which
        gives \(\dim\HH^n=(a_1-1)\cdots(a_{n+1}-1)\). Note that this equals the
  Milnor number of the singularity.
\end{rmk}

We now proceed to the specific examples of interest to compute the
contributions explicitly. These examples are:
\begin{align*}
  cA_\enn:&\quad x_1^2 + x_2^2 + x_3^{\enn+1} + x_4^{k(\enn+1)},& k,\enn\geq 1\\
  cD_4:&\quad x_1^2 + x_2^3 + x_3^3 + x_4^{6k},& k\geq 1\\
  cE_6:&\quad x_1^2 + x_2^3 + x_3^4 + x_4^{12k},& k\geq 1\\
  cE_8:&\quad x_1^2 + x_2^3 + x_3^5 + x_4^{30k},& k\geq 1
\end{align*}
In all cases, the \(x_4 = 0\) slice has an ADE singularity at the
origin, having the type indicated. The 3-folds admit small resolutions
which fully resolve the singularity of the slice; this follows from
{\cite[Satz 0.2]{Brieskorn}} because the exponent of \(x_4\) is a
multiple of the Coxeter number of the ADE singularity.

\begin{thm}\label{thm:bp}
  For each \(\w\) below, we will compute
  \(\HH^*=\HH^*(\mf(\A^5,\Gamma_{\w},\w))\). In all cases,
  \(\HH^d\) vanishes when \(d=2\) or \(d\geq 4\).
  \begin{enumerate}
  \item {\bf Type \(cA_{\enn}\):} Let \(\w = x_1^2 + x_2^2 + x_3^{\enn+1} + x_4^{k(\enn+1)}\). Then
    \begin{align*}\dim\HH^3&=\enn(k(\enn+1)-1), & \dim\HH^{d}&=\enn\mbox{ for }d\leq 1.\end{align*}
  \item {\bf Type \(cD_4\):} Let \(\w = x_1^2 + x_2^3 + x_3^3 + x_4^{6k}\). Then
    \begin{align*}\dim\HH^3&=24k-4, & \dim\HH^{d}&=4\mbox{ for }d\leq 1.\end{align*}
  \item {\bf Type \(cE_6\):} Let \(\w = x_1^2 + x_2^3 + x_3^4 + x_4^{12k}\). Then
    \begin{align*} \dim\HH^3&=72k-6, & \dim\HH^{d}&=6\mbox{ for }d\leq 1.\end{align*}
  \item {\bf Type \(cE_8\):} Let \(\w = x_1^2 + x_2^3 + x_3^5 + x_4^{30k}\). Then
    \begin{align*} \dim\HH^3&=240k-8, & \dim\HH^{d}&=8\mbox{ for }d\leq 1.\end{align*}
  \end{enumerate}
\end{thm}
\begin{proof}
  In each case, the \(\HH^3\) contributions come from Lemma
  \ref{lma:-1}. We will consider the contributions from
  \(m_A(b_0)\) with \(b_0 \geq 0\).

  In the various cases we will use Euclid's algorithm to write:

  \begin{center}
    \begin{tabular}{llll}
      \toprule
      Type & \(b_0=\) & \(q\in\) & \(r\in\) \\
      \midrule
      \(cA_{\enn}\) & \(k(\enn+1)p+(\enn+1)q+r\) & \(\{0,1,\ldots,k-1\}\) &  \(r\in\{0,1,\ldots,\enn\}\)\\
      \(cD_4\) & \(6kp+6q+r\) & \(\{0,1,\ldots,k-1\}\) &  \(r\in\{0,1,\ldots,5\}\)\\
      \(cE_6\) & \(12kp+12q+r\) & \(\{0,1,\ldots,k-1\}\) &  \(r\in\{0,1,\ldots,11\}\)\\
      \(cE_8\) & \(30kp+30q+r\) & \(\{0,1,\ldots,k-1\}\) &  \(r\in\{0,1,\ldots,29\}\)\\
      \bottomrule
    \end{tabular}
  \end{center}

  In the following tables, we indicate: the type \(A\) monomials
  \(m_A(b_0)\); the \(\gamma\) for which \((\gamma, m_A(b_0))\)
  contribute to \(\HH^*\); the number of such \(\gamma\); and the degree
  of \(\HH^*\) to which they contribute. We omit monomials \(m\)
  for which there are no \(\gamma\) such that \((\gamma, m)\)
  contributes.

  In every case, we will see that \(\HH^*\) has the rank stated in the
  theorem in every even degree \(d\leq 0\). The type \(B\)
  contributions, other than those appearing in Lemma \ref{lma:-1},
  will differ only in replacing \(m_A(b_0)\) with
  \(x_0x_0^\vee m_A(b_0)\) and yield the same ranks in every odd
  degree \(d\leq 1\).

    \renewcommand\arraystretch{1.1}

    \begin{table}[htb!]  
      \caption{Table for \(cA_{\enn}\). Note that
        \(b_0=k(\enn+1)p+(\enn+1)q+r\) with \(p\geq 0\),
        \(0\leq q\leq k-1\), \(0\leq r\leq \enn\). The top two row give
        us rank \(\enn\) in every degree \(d\neq -2k\mod 2(k+1)\). The
        bottom row gives us rank \(\enn\) in degrees \(d=-2k\mod 2(k+1)\)
        (\(\omega\) is a chosen primitive \((\enn+1)\)th root of unity
        and \(a\in\{1,2,\ldots,\enn\}\)).}
      \label{fig:ca_odd}
    
      \centering
      \begin{tabular}{ccllcl}
        \toprule
        \(q\) & \(r\) & \(m_A(b_0)\) & \(\gamma\) & \# \(\gamma\) &\(\HH^*\) degree \\
        \midrule
        any & \(< \enn\) & \(x_0^{b_0} x_3^r x_4^{(\enn+1)q+r}\begin{cases}
          1&\mbox{ if }b_0\mbox{ even}\\
          x_1^\vee x_2^\vee&\mbox{ if  }b_0\mbox{ odd}\end{cases}\)
                      & \(\begin{cases}
                        (1,1,1,1) \\
                       (-1,-1,1,1)
                      \end{cases}\) & \(1\) & \(-2(k+1)p-2q\)\\
        \(k-1\) & \(\enn\) & \(x_0^{b_0} x_3^\vee x_4^\vee\begin{cases}
          1&\mbox{ if }b_0\mbox{ even}\\
          x_1^\vee x_2^\vee&\mbox{ if  }b_0\mbox{ odd}\end{cases}\)& \(\begin{cases}(-1,-1,\omega^a,\omega^{-a})\\ (1,1,\omega^a,\omega^{-a})\end{cases}\) & \(\enn\) & \(-2(k+1)p-2k\) \\
        \bottomrule
      \end{tabular}
    \end{table}

    \begin{table}[htb!]  
      \caption{Table for \(cD_4\). The top three rows give us rank
        \(4\) in every degree \(d\neq -2k\mod 2(k+1)\). The bottom row
        gives us rank \(4\) in degrees \(d=-2k\mod 2(k+1)\) (on the
        second and fourth rows, \(\omega\) is a chosen primitive 3rd
        root of unity and \(a,b\in\{1,2\}\)).}
    \label{fig:cd4}
    
    \centering
    \begin{tabular}{cclccl}
      \toprule
      \(q\) & \(r\) & \(m_A(b_0)\) & \(\gamma\) & \# \(\gamma\) &\(\HH^*\) degree \\
      \midrule
      any & \(0\) & \(x_0^{6kp+6q}x_4^{6q}\) & \((1,1,1,1)\) & \(1\) & \(-2(k+1)p-2q\)\\
      any &\(2\) & \(x_0^{6kp+6q+2} x_4^{6q+2}x_2^\vee x_3^\vee\) & \((1,\omega^a,\omega^{-a},1)\), & \(2\) & \(-2(k+1)p-2q\)\\
      any &\(4\) & \(x_0^{6kp+6q+4}x_2x_3x_4^{6q+4}\) & \((1,1,1,1)\) & \(1\) & \(-2(k+1)p-2q\)\\
      \(k-1\) & \(5\) & \(x_0^{6kp+6k-1}x_1^\vee x_2^\vee x_3^\vee x_4^\vee\)& \((-1,\omega^a,\omega^b,-\omega^{-a-b})\) & \(4\) &    \(-2(k+1)p-2k\) \\
      \bottomrule
    \end{tabular}
  \end{table}

  \begin{table}[htb!]
    \caption{Table for \(cE_6\). The top six rows give us rank \(6\)
      in every degree \(d\neq -2k\mod 2(k+1)\). The bottom row gives
      us rank \(6\) in degrees \(d=-2k\mod 2(k+1)\) (\(\omega\) and
      \(i\) are chosen primitive 3rd and 4th roots of unity; and
      \(a\in\{1,2\}\) and \(b\in\{1,2,3\}\)).}
    \label{fig:ce6}
    \centering
    \begin{tabular}{cclccl}
      \toprule
      \(q\) & \(r\) & \(m_A(b_0)\) & \(\gamma\) & \# \(\gamma\) &\(\HH^*\) degree \\
      \midrule
      any & \(0\) & \(x_0^{12kp+12q}x_4^{12q}\) & \((1,1,1,1)\) & \(1\) & \(-2(k+1)p-2q\)\\
      any &\(3\) & \(x_0^{12kp+12q+3}x_4^{12q+3}x_1^\vee x_3^\vee\) & \((-1,1,-1,1)\) & \(1\) & \(-2(k+1)p-2q\) \\
      any &\(4\) & \(x_0^{12kp+12q+4}x_2x_4^{12q+4}\) & \((1,1,1,1)\) & \(1\) & \(-2(k+1)p-2q\)\\
      any &\(6\) & \(x_0^{12kp+12q+6}x_3^2x_4^{12q+6}\) & \((1,1,1,1)\) & \(1\) & \(-2(k+1)p-2q\)\\
      any &\(7\) & \(x_0^{12kp+12q+7}x_2x_4^{12q+7}x_1^\vee x_3^\vee\) & \((-1,1,-1,1)\) & \(1\) & \(-2(k+1)p-2q\)\\
      any &\(10\) & \(x_0^{12kp+12q+10}x_2x_3^2x_4^{12q+10}\) & \((1,1,1,1)\) & \(1\) & \(-2(k+1)p-2q\)\\
      \(k-1\) & \(11\) & \(x_0^{12kp+12k-1}x_1^\vee x_2^\vee x_3^\vee x_4^\vee\)& \((-1,\omega^a,i^b,-\omega^{-a}i^{-b})\) & \(6\) &  \(-2(k+1)p-2k\) \\
      \bottomrule
    \end{tabular}
  \end{table}

  \begin{table}[htb!]  
    \caption{Table for \(cE_8\). The top eight rows give us rank \(8\)
      in every degree \(d\neq -2k\mod 2(k+1)\). The bottom row gives
      us rank \(8\) in degrees \(d=-2k\mod 2(k+1)\) (\(\omega\) and
      \(\zeta\) are chosen primitive 3rd and 5th roots of unity,
      \(a\in\{1,2\}\), and \(b\in\{1,2,3,4\}\)).}
    \label{fig:ce8}
    \centering
    \begin{tabular}[t]{cclccl}
      \toprule
      \(q\) & \(r\) & \(m_A(b_0)\) & \(\gamma\) & \# \(\gamma\) &\(\HH^*\) degree \\
      \midrule
      any & \(0\) & \(x_0^{30kp+30q}x_4^{30q}\) & \((1,1,1,1)\) & \(1\) & \(-2(k+1)p-2q\)\\
      any & \(6\) & \(x_0^{30kp+30q+6}x_3x_4^{30q+6}\) & \((1,1,1,1)\) & \(1\) & \(-2(k+1)p-2q\)\\
      any & \(10\) & \(x_0^{30kp+30q+10}x_2x_4^{30q+10}\) & \((1,1,1,1)\) & \(1\) & \(-2(k+1)p-2q\)\\
      any & \(12\) & \(x_0^{30kp+30q+12}x_3^2x_4^{30q+12}\) & \((1,1,1,1)\) & \(1\) & \(-2(k+1)p-2q\)\\
      any & \(16\) & \(x_0^{30kp+30q+16}x_2x_3x_4^{30q+16}\) & \((1,1,1,1)\) & \(1\) & \(-2(k+1)p-2q\)\\
      any & \(18\) & \(x_0^{30kp+30q+18}x_3^3x_4^{30q+18}\) & \((1,1,1,1)\) & \(1\) & \(-2(k+1)p-2q\)\\
      any & \(22\) & \(x_0^{30kp+30q+22}x_2x_3^2x_4^{30q+22}\) & \((1,1,1,1)\) & \(1\) & \(-2(k+1)p-2q\)\\
      any & \(28\) & \(x_0^{30kp+30q+28}x_2x_3^3x_4^{30q+28}\) & \((1,1,1,1)\) & \(1\) & \(-2(k+1)p-2q\)\\
      \(k-1\) & \(29\) & \(x_0^{30kp+30k-1}x_1^\vee x_2^\vee x_3^\vee x_4^\vee\)& \((-1,\omega^a,\xi^b,-\omega^{-a}\xi^{-b})\) & \(8\) &  \(-2(k+1)p-2k\)\qedhere \\
      \bottomrule
    \end{tabular}
  \end{table}

\ \\
\ \\
\ \\
\ \\
\ \\
\ \\
\ \\
\ \\
\ \\
\end{proof}

\subsection{Laufer's examples}
Let 
\[ \check{\mathbf{w}} = x_1^3 + x_1 x_2^{2k+1} + x_2 x_3^2 + x_4^2\]
This polynomial defines a \(cD_4\) singularity: the \(x_1=x_2\) slice
has an isolated \(D_4\) singularity at the origin. Laufer
\cite{Laufer} showed that this admits a small resolution with a single
exceptional curve; the small resolution yields a partial resolution of
the \(D_4\) slice (the map from the minimal resolution to the partial
resolution collapses the three peripheral curves in the \(D_4\)
configuration).

The Berglund-H\"{u}bsch transpose is
\[ \mathbf{w} = x_1^3 x_2 + x_2^{2k+1} x_3 + x_3^2 + x_4^2\]
which has
\[ \Gamma_{\w} = \{ (t_0,t_1,t_2,t_3,t_4) : t_1^3 t_2 = t_2^{2k+1} t_3 =
  t_3^2 =t_4^2 = t_0t_1t_2t_3t_4 \}.\]

\begin{lem}
  There is a \(3\)-to-\(1\) surjective homomorphism
  \begin{align*}
    T&\colon \bm{\mu}_2\times \bm{\mu}_3\times\C^{\times}\to\Gamma_{\w},\\
    T&(s, \mu, \tau) = (s\mu^{-1} \tau^{-(4k+4)}, \mu\tau^{4k+1},
       \tau^3, \tau^{6k+3}, s\tau^{6k+3}).
  \end{align*}
  The composition \(\chi\circ T\) is given by
  \((s,\mu,\tau)\mapsto\tau^{12k+6}\).
\end{lem}
\begin{proof}
  We first show that the stated homomorphism is surjective. Since
  \(t_2^{2k+1}t_3=t_3^2\) we get \(t_3=t_2^{2k+1}\). Since
  \(t_4^2=t_3^2\), we have \(t_4=\pm t_2^{2k+1}\). Since
  \(t_1^3t_2 = t_3^2=t_2^{4k+2}\), we get \(t_1^3=t_2^{4k+1}\). If
  \(t_2=\tau^3\) for some \(\tau\in\C^\times\) then
  \(t_1=\mu\tau^{4k+1}\) for some cube root \(\mu\) of unity. Finally,
  \(t_0\) is determined by \(t_0\cdots t_4=t_3^2\), which gives
  \(t_0=\pm\mu^{-1}\tau^{-4k-4}\).

  To see that the homomorphism is \(3\)-to-\(1\), observe that its
  kernel consists of triples \((s, \mu, \tau)\) such that
  \[\mu\tau^{4k+1}=\tau^3=\tau^{6k+3}=s\tau^{6k+3}=1.\] In
  particular, this means \(s = 1\) and \(\tau^3=1\). The condition
  \(1=\mu\tau^{4k+1}=\mu\tau^{k+1}\) means that \(\mu=\tau^{-k-1}\),
  so the kernel is \(\{(1,\tau^{-k-1},\tau)\ :\ \tau^3=1\}\), which
  has size 3.
\end{proof}

The kernel \(\ker(\chi\circ T)\) is then
\(\bm{\mu}_2\times\bm{\mu}_3\times\bm{\mu}_{12k+6}\); recall that
\(T\) is \(3\)-to-\(1\), so this is three times the size of
\(\ker\chi\). We now identify which combinations of fixed and unfixed
variables are possible for \(\gamma\in\ker\chi\).

\begin{lem}
  Let \(\gamma= T(s,\mu,\tau)\in\ker\chi\). The possible
  combinations of fixed and unfixed variables for \(\gamma\) are
  given by the table below. We state the conditions on
  \((s, \mu, \tau) \in \bm{\mu}_2 \times \bm{\mu}_3 \times
  \bm{\mu}_{12k+6}\) such that \(\gamma= T(s,\mu,\tau)\) fixes
  this combination of variables, and also the number of such
  \(\gamma\) (remembering that \(T\) is \(3\)-to-\(1\)).
  
  \centering
  \begin{tabular}[h]{lc|rll}
    \toprule
    Fixed variables & Number of \(\gamma=T(s, \mu, \tau) \in\ker\chi\) & \(s\) & \(\mu\) & \(\tau\) \\
    \midrule
    \(\{0,1,2,3,4\}\) & \(1\)                                          & \(1\) & \(\tau^{-4(k+1)}\) & \(\tau^3=1\) \\
    \(\{0\}\)         & \(1\)                                          & \(1\) & \(\tau^{-4(k+1)}\) & \(\tau^3=-1\) \\
    \(\{2,3\}\)       & \(2\)                                          & \(-1\)  & \(\mu\neq\tau^{-k-1}\) & \(\tau^3=1\) \\
    \(\emptyset\)     & \(6k+2\)                                       & \multicolumn{3}{c}{See below} \\
    \midrule
    \multicolumn{5}{l}{Five further cases which do not contribute to \(\HH^*\):
    \(\{3\}\), \(\{4\}\), \(\{3,4\}\), \(\{1,2,3\}\) and \(\{2,3,4\}\)} \\
    \bottomrule
  \end{tabular}
\end{lem}
\begin{proof}
  Let \(\gamma=T(s, \mu, \tau)\) with \(s\in\{\pm 1\}\),
  \(\mu\in\bm{\mu}_3\), \(\tau\in\bm{\mu}_{12k+6}\).

  If \(x_0\) is fixed then \(\tau^{-4(k+1)}=s\mu\). This means that
  \(\tau^{24(k+1)}=1\), but \(\tau^{12k+6}=1\), so
  \(\tau^{6}=1\). Therefore \(\tau^{-4(k+1)}\) is a cube root of
  unity, which means that \(s=1\). This means that the other
  variables transform as \(\mu \tau^{4k+1}=\tau^{-3}\), \(\tau^3\),
  \(\tau^{6k+3}\) and \(\tau^{6k+3}\). There are two possibilities:
  \(\tau^3=1\) (which fixes all variables) or \(\tau^3=-1\) (which
  fixes none).

  If \(x_1\) is fixed then \(\mu=\tau^{-4k-1}\), so
  \(\tau^{-12k-3}=1\), but \(\tau^{12k+6}=1\), so \(\tau^3=1\). This
  means that \(x_2\) and \(x_3\) are also fixed. The variable \(x_0\)
  transforms as \(s\mu^{-1}\tau^{-4k-4}=s\tau^{-3}=s\), so either
  \(x_0\) is fixed (as in the previous case) or \(s=-1\), in which
  case both \(x_0\) and \(x_4\) are unfixed.
  
  If \(x_2\) is fixed then \(\tau^3=1\) so \(\tau^{6k+3}=1\) and
  \(x_3\) is also fixed. If \(x_0\) or \(x_1\) is fixed then we are in
  a previous case; assume they are not. Then \(\mu \neq \tau^{-k-1}\)
  and \(s\) can take on either value because both \(\mu\) and
  \(\tau^{-k-1}\) are in \(\bm{\mu}_3\), so \(\mu = -\tau^{-k-1}\) is
  impossible. If \(s = 1\) then \(x_4\) is fixed (yielding fixed
  variables \(\{2, 3, 4\}\)); otherwise we get fixed variables
  \(\{2, 3\}\).

  Finally, if none of \(x_0\), \(x_1\), \(x_2\) are fixed then the
  remaining variables can be fixed in any combination. We will see in
  Theorem \ref{thm:laufer} that the only such \(\gamma\) which
  contribute \(\gamma\)-monomials to \(\HH^*\) are those which fix no
  variables. There are \(6k+2\) of these. To see this, we argue as
  follows. If \(x_3\) is not fixed then \(\tau^{6k+3}\neq 1\), so
  \(\tau^{6k+3}=-1\). If \(x_4\) is not fixed then \(s\tau^{6k+3}=-1\)
  means that \(s = 1\). The remaining conditions become
  \[\mu\neq \tau^{-4(k+1)},\qquad\mu\neq \tau^{-4k-1}.\] The second
  condition always holds because \(\tau^{6k+3}=-1\), so
  \(\tau^{-4k-1}\) is not a cube root of unity
  (\((\tau^{-4k-1})^3=\tau^{-12k-3}=\tau^3\neq 1\)). The first
  condition implies \(1=(\tau^{-4(k+1)})^3=\tau^{12k+12}=\tau^6\),
  which can hold only if \(\tau^3=-1\). Therefore there are \(6k\)
  roots of \(\tau^{6k+3}=-1\) for which \(\mu\) can take on any value
  and \(3\) roots of \(\tau^3=-1\) for which \(\mu\) can be two out of
  the three roots of unity. This gives a total of \(3(6k+2)\)
  combinations \((1,\mu,\tau)\), and this triple-counts the available
  \(\gamma\)s because \(T\) is \(3\)-to-\(1\).
\end{proof}

We now pick the following monomial bases \(J_\gamma\) for the relevant
Jacobian rings:
\begin{align*}
  \mathrm{Jac}(\mathbf{w}_{|x_1,x_2,x_3,x_4}) &= \mathrm{Jac}(\mathbf{w}_{|x_1,x_2,x_3}) \\
  & = \mathbb{C}[x_1,x_2,x_3]/ (3x_1^2 x_2, x_1^3 + (2k+1)x_2^{2k} x_3, x_2^{2k+1} + 2x_3) \\
  &= \mathbb{C} \{ 1,x_2,x_2^2,\ldots,x_2^{4k+1},x_1,x_1x_2, x_1x^2_2,\ldots, x_1x_2^{4k+1},x_1^2 \} \\ 
  \mathrm{Jac}(\mathbf{w}_{|x_2,x_3,x_4}) &= \mathrm{Jac}(\mathbf{w}_{|x_2,x_3})\\
  &= \mathbb{C}[x_2,x_3]/ ((2k+1)x_2^{2k} x_3, x_2^{2k+1} + 2x_3)\\ & = \mathbb{C} \cdot \{ 1, x_2, x_2^2,\ldots,x_2^{4k}\}  \\
  \mathrm{Jac}(\mathbf{w}_{|x_3,x_4}) &= \mathbb{C}[x_3,x_4]/ (2x_3, 2x_4)  = \mathbb{C} \cdot 1\\
  \mathrm{Jac}(\mathbf{w}_{|x_3}) &= \mathrm{Jac}(\mathbf{w}_{|x_4}) = \mathbb{C} \cdot 1
\end{align*}

\begin{thm}\label{thm:laufer}
  If \(\mathbf{w} = x_1^3 x_2 + x_2^{2k+1} x_3 + x_3^2 + x_4^2\) then
  \(\HH^*(\A^5,\Gamma_{\w},\w)\) satisfies
  \begin{align*}\dim\HH^3&=6k+5, & \dim\HH^{d}&=1\mbox{ for }d\leq 1\end{align*}
  and \(\dim\HH^d=0\) for \(d=2\) and \(d\geq 4\).

  The \(\HH^*\) contributions for these singularities are as follows:

  \centering
  
  \begin{tabular}[h]{llll}
    \toprule
          Monomial & Type & Degree in \(\HH^*\) & Number of contributions\\
    \midrule
    \(x_0^\vee x_1^\vee x_2^\vee x_3^\vee x_4^\vee\) & C & \(3\) & \(6k+2\)  \\
    \(x_0^\vee x_1^\vee x_2^\vee x_3^\vee x_4^\vee\) & B & \(3\) & \(1\)  \\
    \(x_0^\vee x_1^\vee x_2^{2k} x_4^\vee\) & C & \(3\) & \(2\) \\
    \midrule
    \(x_0^{(6k+3)p-1} x_1^\vee x_2^\vee x_3^\vee x_4^\vee\) & A & \(-4(k+1)p+2\) & \(p\geq 1\) even\\    
    \(x_0^{(6k+3)p+3q}x_2^{2q}\) & A & \(-4(k+1)p-2q\) & \(p\geq 0\), \(0\leq q\leq 2k\), \(p=q\mod 2\)\\
    \(x_0^{(6k+3)p+3q+4}x_1x_2^{2q+3}\) & A & \(-4(k+1)p-2q-2\) & \(p\geq 0\), \(0\leq q\leq 2k-1\), \(p=q\mod 2\)\\
    \(x_0^{(6k+3)p+6k+4}x_1x_2\) & A & \(-4(k+1)(p+1)\) & \(p\geq 0\) even\\
    \(x_0^{(6k+3)p+6k+2}x_1^2\) & A & \(-4(k+1)p-4k-2\) & \(p\geq 0\) even \\
    \midrule
    \multicolumn{4}{l}{Type B contributions in \(\HH^{d+1}\) for each type A monomial contributing to \(\HH^{d}\).} \\
    \bottomrule
  \end{tabular}  
\end{thm}

\begin{proof}
  Assuming the stated monomials are correct, the patient reader can
  check that every degree less than or equal to \(1\) picks up
  precisely one contribution as \(p\) and \(q\) vary (it suffices to
  check this over the degree range from \(0\) to \(-8(k+1)\)). We will
  therefore focus on establishing the list of contributing
  \(\gamma\)-monomials.
  
  We work one set of fixed variables at a time and figure out which
  \(\gamma\)-monomials can contribute.

  \(\{0, 1, 2, 3, 4\}\): The possible \(A\)-type monomials are as
  follows.
  
  \(x_0^{b_0}x_2^{b_2}\) with \(0\leq b_2\leq 4k+1\). For this
  to contribute, we need the existence of a \(u\in\Z\) such that
  \(s^{b_0}\mu^{-b_0}\tau^{3b_2-4(k+1)b_0}=\tau^{(12k+6)u}\) for
  all
  \((s,\mu,\tau)\in\bm{\mu}_2\times\bm{\mu}_3\times\bm{\mu}_{12k+6}\). By
  taking \((s,\mu,\tau)=(-1,e^{2\pi i/3},1)\) we see that
  \(b_0=0\mod 6\), which leaves
  \(\tau^{3b_2-4(k+1)b_0}=\tau^{(12k+6)u}\), so
  \(3b_2-4(k+1)b_0=(12k+6)u\). If we write \(b_0=6\beta_0\) and
  \(2\beta_0=(2k+1)p+q\) with \(p\geq 0\),
  \(q\in\{0,1\ldots,2k\}\), \(p=q\mod 2\), then we get
  \[b_2 = (4k+2)(2(k+1)p+u)+4(k+1)q,\] so if we reduce modulo
  \(4k+2\) we get \(b_2=2q\mod 4k+2\). Since
  \(0\leq b_2\leq 4k+1\) and \(q\leq 2k\), this determines
  \(b_2\). The result is a contribution
  \(x_0^{(6k+3)p+3q}x_2^{2q}\in \HH^{-4(k+1)p-2q}\) for all
  \(p\geq 0\), \(q\in\{0,1,\ldots,2k\}\) with \(p=q\mod 2\).

  \(x_0^{b_0}x_1x_2^{b_2}\) with \(0\leq b_2\leq 4k+1\). For this to
  contribute, we need \(b_0=1\mod 3\), \(b_0=0\mod 2\) (so
  \(b_0=4\mod 6\)) and \(3b_2+4k+1-4(k+1)b_0=(12k+6)u\) for some
  \(u\in\Z\). As in the previous case, we write \(b_0=6\beta_0+4\) and
  \(2\beta_0=(2k+1)p+q\) with \(p\geq 0\), \(q\in\{0,1,\ldots,2k\}\),
  \(p=q\mod 2\). Arguing as before, we deduce that \(b_2=2q+3\mod
  4k+2\). This means \(b_2=2q+3\) except in the case \(q=2k\), when
  \(b_2=1\). We get contributions
  \(x_0^{(6k+3)p+3q+4}x_1x_2^{2q+3}\in\HH^{-4p(k+1)-2q-2}\) for
  \(p\geq 0\), \(q\in\{0,1,\ldots,2k-1\}\) with \(p=q\mod 2\) and a
  contribution \(x_0^{(6k+3)p+6k+4}x_1x_2\in\HH^{-4(p+1)(k+1)}\).

  \(x_0^{b_0}x_1^2\). For this to contribute, we need \(b_0=0\mod 2\),
  \(b_0=2\mod 3\) (so \(b_0=2\mod 6\)) and \(8k+2-4(k+1)b_0=(12k+6)u\)
  for some \(u\in\Z\). If we write \(b_0=6\beta_0+2\) and
  \(2\beta_0=(2k+1)p+q\) with \(p\geq 0\), \(q\in\{0,1,\ldots,2k\}\),
  \(p=q\mod 2\) then we get \(q=2k\mod 2k+1\) and so
  \(x_0^{(6k+3)p+6k+2}x_1^2\in\HH^{-4(k+1)p-4k-2}\) (with \(p\) even).

  We also get corresponding \(B\)-type monomials by replacing
  \(x_0^{b_0}\) with \(x_0^{b_0+1}x_0^\vee\).

  \(\{0\}\): Any type \(A\) contribution is
  \(x_0^{b_0}x_1^\vee\cdots x_4^\vee\). This transforms as
  \begin{align*}
    \xi\left(x_0^{b_0}x_1^\vee\cdots x_4^\vee\right)(s,\mu,\tau)
    &=
      (s\mu^{-1}\tau^{-4(k+1)})^{b_0}(\mu\tau^{4k+1})^{-1}\tau^{-3}\tau^{-(6k+3)}s^{-1}\tau^{-(6k+3)}\\
    &= s^{b_0-1}\mu^{-b_0-1}\tau^{-2(6k+3)-3-(4k+1)-4(k+1)b_0}.\end{align*}
  For this to coincide with \(\tau^{(12k+6)u}\) for all
  \((s,\mu,\tau)\in\bm{\mu}_2\times\bm{\mu}_3\times\bm{\mu}_{12k+6}\)
  we need
  \[b_0=1\mod 2,\quad b_0=-1\mod 3\Rightarrow b_0=5\mod 6\] and
  \(-2(6k+3) -3 - (4k+1) - 4(k+1)b_0 = (12k+6)u\) for some
  \(u\in\Z\). Write \(b_0=6\beta_0-1\). Then we get
  \(-4(k+1)\beta_0 = (2k+1)(u+1)\). Since \(\gcd(4(k+1), 2k+1)=1\), we
  deduce that \(\beta_0 = (2k+1)p\) and \(u+1 = -4(k+1)p\) for some
  \(p\). In other words, we get
  \(x_0^{(12k+6)p-1}x_1^\vee\cdots x_4^\vee\in\HH^{2-4(k+1)p}\)
  (\(p\geq 1\)). There is a corresponding \(B\)-type monomial
  \(x_0^{(12k+6)p}x_0^\vee x_1^\vee\cdots x_4^\vee\in\HH^{3-4(k+1)p}\)
  (\(p\geq 0\)).

  \(\{1, 2, 3\}\): The possible \(\gamma\)-monomials are of type
  \(C\). They have the form \(x_0^\vee x_1^{b_1} x_2^{b_2} x_4^\vee\)
  with \(b_1=0,1\), \(b_2=0,1,\ldots,4k+1\) or \(b_1=2\),
  \(b_2=0\). If this contributes then we have \(b_1=b_0=-1\mod 3\),
  which leaves the only possibility as \(x_0^\vee x_1^2
  x_4^\vee\). This transforms under the action of \(T(1,1,\tau)\) as
  \(\tau^{6k+3}\), which is not an integer power of \(\tau^{12k+6}\),
  so this monomial does not contribute.

  \(\{2, 3, 4\}\): The possible \(\gamma\)-monomials are
  \(x_0^\vee x_1^\vee x_2^{b_2}\in C_\gamma\), which transform
  nontrivially under the action of \(T(-1,1,1)\) and hence do not
  contribute to \(\HH^*\).

  \(\{2, 3\}\): There are two \(\gamma\) fixing precisely
  \(x_2,x_3\). The only \(\gamma\)-monomials are
  \(x_0^\vee x_1^\vee x_2^{b_2} x_4^\vee\) with
  \(b_2=0,1,\ldots,4k\). These transform according to the character
  \(\tau^{3b_2-6k}\), which is an integer power of \(\tau^{12k+6}\) if
  and only if \(b_2=2k\). This yield two contributions
  \((\gamma, x_0^\vee x_1^\vee x_2^{2k} x_4^\vee)\in\HH^3\).

  \(\{3, 4\}\): The only \(\gamma\)-monomial is
  \(x_0^\vee x_1^\vee x_2^\vee\in C_\gamma\), which transforms
  nontrivially under the action of \(T(-1,1,1)\) and hence does not
  contribute to \(\HH^*\).

  \(\{3\}\): The only \(\gamma\)-monomial
  \(x_0^\vee x_1^\vee x_2^\vee x_4^\vee\in C_\gamma\) which transforms
  as \(\tau^{6k+3}\) under the action of \(T(1,1,\tau)\) and hence
  does not contribute to \(\HH^*\).

  \(\{4\}\): The only \(\gamma\)-monomial
  \(x_0^\vee x_1^\vee x_2^\vee x_3^\vee\in C_\gamma\) which transforms
  as \(\tau^{6k+3}\) under the action of \(T(1,1,\tau)\) and hence
  does not contribute to \(\HH^*\).

  \(\emptyset\): The \(C\)-type monomial
  \(x_0^\vee\cdots x_4^\vee\in\HH^3\) contributes whenever \(\gamma\)
  has no fixed variables; there are precisely \(6k+2\) such elements
  \(\gamma\).
\end{proof}

\subsection{More \(cA_\enn\) examples}

By {\cite[Theorem 1.1]{Katz}}, any \(cA_\enn\) singularity with a
small resolution is given by an equation \(x_1^2+x_2^2+f(x_3, x_4)=0\)
where germ of the plane curve \(f=0\) at the origin has \(\enn+1\)
distinct smooth branches, and conversely, any such singularity admits
a small resolution (the converse was also proved in
{\cite[p. 676]{Friedman}}).

Let
\[\w=\check{\w}=x_1^2 + x_2^2 + x_3x_4(x_3^{\enn-1} + x_4^{k(\enn-1)}).\]
The singularity \(\check{\w} = 0\) is of type \(cA_\enn\): the
\(x_3=x_4\) slice has an \(A_\enn\) singularity at the origin. The curve
\(x_3x_4(x_3^{\enn-1} + x_4^{k(\enn-1)}) = 0\) has multiplicity \(\enn+1\) and
\(\enn+1\) distinct branches at the origin:
\[x_3=0,\quad x_4=0,\quad\mbox{ and }\quad x_3 + \mu x_4^k = 0\mbox{
    for } \mu^{\enn-1} = -1.\] Therefore, this singularity admits a small
resolution.

\begin{lem}
  There exists a surjective \(2\)-to-\(1\) homomorphism
  \(T\colon\bm{\mu}_2\times\bm{\mu}_{2(\enn-1)}\times\C^\times\to\Gamma_{\w}\)
  which we will construct in the proof. The composition
  \(\chi\circ T\) is given by
  \((\pm 1,\sigma,\tau)\mapsto\sigma^2\tau^{2k\enn+2}\).
\end{lem}
\begin{proof}
  The group \(\Gamma_{\w}\) is defined by the equations
  \[t_0t_1t_2t_3t_4 = t_1^2 = t_2^2 = t_3^\enn t_4 = t_3 t_4^{k(\enn-1)+1},\]
  which imply \(t_3^{\enn-1}=t_4^{k(\enn-1)}\), so \(t_3=\xi t_4^k\) for
  some \(\xi\) with \(\xi^{\enn-1}=1\). Substituting back, we get
  \[t_1^2 = t_2^2 = \xi t_4^{k\enn+1}.\] Pick a square root \(\sigma\)
  for \(\xi\) and a square root \(\tau\) for \(t_4\) such that
  \(t_1 = \sigma\tau^{k\enn+1}\); then \(t_2 = \pm\sigma\tau^{k\enn+1}\),
  \(t_3 = \sigma^2\tau^{2k}\), \(t_4=\tau^2\),
  \(t_0=\pm\sigma^{-2}\tau^{-2(k+1)}\). This shows that the
  homomorphism
  \[T(\pm 1, \sigma, \tau) = (\pm\sigma^{-2}\tau^{-2(k+1)},
    \sigma\tau^{k\enn+1}, \pm\sigma\tau^{k\enn+1}, \sigma^2\tau^{2k},
    \tau^2)\] is surjective. To see that it is \(2\)-to-\(1\), note
  that its kernel consists of triples \((1,\sigma,\tau)\) for which
  \(\tau^2=1\) (so \(\tau=\pm 1\)) and \(\sigma=\tau^{-k\enn-1}\). This
  has size 2.
\end{proof}

The kernel \(\ker(\chi\circ T)\) is the subgroup
\[\{(\pm 1,\sigma,\tau) \in
  \bm{\mu}_2\times\bm{\mu}_{2(\enn-1)}\times\bm{\mu}_{2(k\enn+1)(\enn-1)}\ :\
  \tau^{2(k\enn+1)}=\sigma^{-2}\}.\] The projection to
\(\tau\in\bm{\mu}_{2(k\enn+1)(\enn-1)}\) is surjective and split by the map
\(\tau\mapsto (1,\tau^{-k\enn-1},\tau)\); its kernel consists of triples
\((\pm 1,\pm 1,1)\), so there is an isomorphism
\begin{align*}
  \bm{\mu}_2\times\bm{\mu}_2\times\bm{\mu}_{2(k\enn+1)(\enn-1)}&\to
  \ker(\chi\circ T)\\
  (s_1,s_2,\tau)&\mapsto (s_1,s_2\tau^{-(k\enn+1)},\tau).
\end{align*}
We will work with elements of this group; since \(T\) is
\(2\)-to-\(1\), this will mean that we overcount contributions to
\(\HH^*\) by a factor of \(2\). We now identify which combinations of
fixed and unfixed variables are possible for \(\gamma\in\ker\chi\).

\begin{lem}
  The possible combinations of fixed variables are given in the table
  below, along with the number of elements \(\gamma\in\ker\chi\) which
  give rise to these fixed variables.
  
  \centering
  \begin{tabular}[h]{ll}
    \toprule
    Fixed variables & \(\# \gamma\) \\
    \midrule
    \(\emptyset\)     & \(k\enn(\enn-1)\)\\
    \(\{0\}\)         & \(\enn-2\) \\
    \(\{0,3,4\}\)     & \(1\) \\
    \(\{0, 1, 2\}\)   & \(\enn-2\) \\
    \(\{0,1,2,3,4\}\) & \(1\) \\
    \midrule
    \multicolumn{2}{l}{The following cases occur, but do not
    contribute to \(\HH^*\):} \\
    \multicolumn{2}{l}{\(\{1\}\), \(\{2\}\), \(\{1,2\}\), \(\{0,1\}\), \(\{0,2\}\), \(\{0,1,3,4\}\), \(\{0,2,3,4\}\)}\\
    \bottomrule
  \end{tabular}
\end{lem}
\begin{proof}
  Let \(\gamma=T(s_1, s_2\tau^{-(k\enn+1)}, \tau)\) with
  \(s_1,s_2\in\{1,-1\}\) and \(\tau\in\bm{\mu}_{2(k\enn+1)(\enn-1)}\).

  If \(x_0\) is fixed then \(s_1\tau^{2(k\enn+1)-2k-2} = 1\), so
  \(\tau^{2(\enn-1)k} = s_1\). Since
  \(\tau^{2(\enn-1)(k\enn+1)} = 1\) this implies
  \(\tau^{2(\enn-1)} = s_1^\enn\), and therefore
  \(s_1^{\enn(k\enn+1)} = 1\). If \(s_1=1\) then this always
  holds. If \(s_1=-1\) then this holds if and only if
  \(\enn(k\enn+1)\) is even. Therefore the element which fix
  \(x_0\) are those of the form
  \(T(1,s_2 \tau^{-(k\enn+1)}, \tau)\) with
  \(\tau^{2(\enn-1)}=1\) and (if \(\enn(k\enn+1)\) is even)
  \(T(-1,s_2\tau^{-(k\enn+1)},\tau)\) with
  \(\tau^{2(\enn-1)}=(-1)^\enn\).

  \(x_1\) is fixed if and only if \(\sigma\tau^{k\enn+1} = s_2 = 1\).

  \(x_2\) is fixed if and only if
  \(s_1\sigma\tau^{k\enn+1} = s_1s_2 = 1\), that is \(s_1 = s_2\).

  \(x_4\) is fixed if and only if \(\tau^2=1\). That is \(\tau=\pm
  1\). 

  \(x_3\) is fixed if and only if
  \(\sigma^2\tau^{2k} = \tau^{-2k\enn-2+2k} = \tau^{-2(k(\enn-1)+1)} =
  1\). Note that
  \[\gcd(k\enn+1,k(\enn-1)+1)=1\quad \mbox{and}\quad\gcd(\enn-1,k(\enn-1)+1)=1,\]
  so the only way we can simultaneously solve
  \(\tau^{2(k\enn+1)(\enn-1)}=1\) and
  \(\tau^{-2(k(\enn-1)+1)} = 1\) is if \(\tau^2=1\). This means
  that \(x_3\) is fixed if and only if \(\tau = \pm 1\) (if and
  only if \(x_4\) is also fixed).

  The following table enumerates the possibilities for
  combinations of fixed variables and the counts\footnote{Recall
    that if we count elements of \(\ker(\chi\circ T)\) then we
    overcount by a factor of 2. We have removed this factor of 2
    in the table.} of \(\gamma\in\ker\chi\) which fix this
  combination of variables (we omit the \(\#\gamma\) data for
  any combinations which turn out not to contribute to
  \(\HH^*\); in particular this allows us to ignore the
  distinction between \(\ell(k\ell+1)\) even/odd).

  \centering
  \begin{tabular}{cccc|ccccc|c}
    \toprule
    \(s_1\) & \(s_2\) & \(\tau^2\) & \(\tau^{2(\enn-1)}\) & \multicolumn{5}{c|}{fixed variables} & \(\#\gamma\) \\
    \midrule
    \(1\)   & \(1\)  & \(1\)            & \(1\)       & \(0\) & \(1\) & \(2\) & \(3\) & \(4\)    & \(1\) \\
            &        & \(\neq 1\)       & \(1\)       & \(0\) & \(1\) & \(2\) &       &          & \(\enn-2\) \\
            &        &                  & \(\neq 1\)  &       & \(1\) & \(2\) &       &          &  \\
    \midrule
    \(1\)   & \(-1\) & \(1\)            & \(1\)       & \(0\) &       &       & \(3\) & \(4\)    & \(1\) \\
            &        & \(\neq 1\)       & \(1\)       & \(0\) &       &       &       &          & \(\enn-2\) \\
            &        &                  & \(\neq 1\)  &       &       &       &       &          & \(k\enn(\enn-1)\) \\
    \midrule
    \(-1\)  & \(1\)  & \(1\)            & \(1\)       & \(0\) & \(1\) &       & \(3\) & \(4\)    &  \\
            &        & \(\neq 1\)       & \(1\)       & \(0\) & \(1\) &       &       &          &  \\
            &        &                  & \((-1)^\enn\)  & \(0\) & \(1\) &       &       &          &  \\
            &        &                  & else        &       & \(1\) &       &       &          &  \\
    \midrule
    \(-1\)  & \(-1\) & \(1\)            & \(1\)       & \(0\) &       & \(2\) & \(3\) & \(4\)    &  \\
            &        & \(\neq 1\)       & \(1\)       & \(0\) &       & \(2\) &       &          &  \\
            &        &                  & \((-1)^\enn\)  & \(0\) &       & \(2\) &       &          &  \\
            &        &                  & else        &       &       & \(2\) &       &          &  \\
    \bottomrule
  \end{tabular}

\end{proof}

We pick the monomial basis \(x_3^a x_4^b\), \(0\leq a\leq \enn-1\),
\(0\leq b\leq k(\enn-1)-1\) for the Jacobian ring of
\(\w|_{x_{j_1}}=\cdots=x_{j_{\enn-k}}=0\) when \(x_3\) and \(x_4\) are
fixed and the monomial basis \(1\) when they are not.

\begin{thm}\label{thm:can}
  If \(\mathbf{w} = x_1^2 + x_2^2 + x_3x_4(x_3^{\enn-1} + x_4^{k(\enn-1)})\) then
  \(\HH^*(\A^5, \Gamma_\w, \w)\) satisfies
  \begin{align*}\dim\HH^3&=(k\enn+1)(\enn-1), & \dim\HH^{d}&=\enn\mbox{ for }d\leq 1\end{align*}
  and \(\dim\HH^d=0\) for \(d=2\) and \(d\geq 4\).

  The \(\HH^*\) contributions for these singularities are given by the
  following table.

  \centering
  
  \begin{tabular}[h]{llll}
    \toprule
    Monomial & Type & Degree in \(\HH^*\) & Number of contributions\\
    \midrule
    \(x_0^\vee x_1^\vee x_2^\vee x_3^\vee x_4^\vee\) & C & \(3\) & \(k\enn(\enn-1)\)  \\
    \(x_0^\vee x_1^\vee x_2^\vee x_3^\vee x_4^\vee\) & B & \(3\) & \(\enn-1\)  \\
    \midrule
    \multirow{3}{*}{\(x_0^{(k\enn+1)p+q \enn+r}vx_3^r x_4^{q(\enn-1)+r}\)} & \multirow{3}{*}{A} & \multirow{3}{*}{\(-2(k+1)p-2q\ \Bigg\{\)} & \(0 \leq q\leq k-1\)\\
             &&& \(0\leq r\leq \enn-1\)\\
             &&& \(p\geq 0\)\\
    \(x_0^{(k\enn+1)p+k\enn}vx_3^{\enn-1}\) & A & \(-2(k+1)p-2k\) & \(p\geq 0\)\\
    \(x_0^{(k\enn+1)p+k\enn}vx_4^{k(\enn-1)}\) & A & \(-2(k+1)p-2k\) & \(p\geq 0\)\\
    \(x_0^{(k\enn+1)p+k\enn}vx_3^{\vee}x_4^{\vee}\) & A & \(-2(k+1)p-2k\) & \(\enn-2\), \(p\geq 0\)\\
    \midrule
    \multicolumn{4}{l}{Type B contributions in \(\HH^{d+1}\) for each type A monomial contributing to \(\HH^d\)}\\
    \bottomrule
  \end{tabular}
  
  \flushleft In this table, we have written
  \[v=\begin{cases}1&\mbox{ if }b_0=0\mod 2, \\x_1^\vee
    x_2^\vee&\mbox{ if } b_0=1\mod 2.\end{cases}\]
\end{thm}
\begin{proof}  
  For each \(\gamma\)-monomial \(m\), let \(b_0,\ldots,b_4\) be the
  total exponents of \(x_0,\ldots,x_4\) in \(m\). This monomial
  transforms under \(T(s,\sigma,\tau)\) as
  \[s^{b_0+b_2} \sigma^{-2b_0+b_1+b_2+2b_3}
    \tau^{-2b_0(k+1)+(k\enn+1)(b_1+b_2)+2kb_3+2b_4},\] which
  agrees with \((\chi\circ T)^{\otimes u}(s,\sigma,\tau)\) for
  all \((s,\sigma,\tau)\) if and only if
  \begin{align}
    \label{eq:can_a}b_0&=b_2\mod 2,\\
    \label{eq:can_b}b_1+b_2+2b_3&=2b_0+2u\mod 2(\enn-1),\\
    \label{eq:can_c}(k\enn+1)(b_1+b_2)+2kb_3+2b_4&=2(k\enn+1)u+2b_0(k+1).
  \end{align}
  Reducing Equation \eqref{eq:can_b} modulo \(2\) tells us that
  \(b_1=b_2\mod 2\). For \(i=1,2\), the only possibilities for \(b_i\)
  are \(0\) (if \(x_i\) is fixed) or \(-1\) (if \(x_i\) is not
  fixed). Thus, if \((\gamma, m)\) contributes to \(\HH^*\) then either
  \(x_1\) and \(x_2\) are both fixed or neither is fixed. This
  immediately rules out the contributions from \(\gamma\) with fixed
  variables \(\{1\}\), \(\{2\}\), \(\{0,1\}\), \(\{0,2\}\)
  \(\{0,1,3,4\}\), \(\{0,2,3,4\}\). Moreover, if \(x_2\) is fixed then
  \(b_2=0\) so \(b_0=0\) so \(x_0\) must also be fixed (or else we
  would have \(b_0=-1\)). This rules out contributions with fixed
  variables \(\{1,2\}\).

  We now dispose of the type \(C\) contributions. These come from
  \((\gamma,x_0^\vee x_1^\vee x_2^\vee x_3^\vee x_4^\vee)\in \HH^3\)
  where \(\gamma\) fixes no variables; there are \(k\enn(\enn-1)\) of these.
  
  Since \(b_1=b_2\mod 2\) and \(b_1,b_2\in\{0,-1\}\), Equations
  \eqref{eq:can_b} and \eqref{eq:can_c} become
  \begin{align}
    \label{eq:can_1} b_3 & = b_0 + u - b_1 \mod \enn-1\\
    \label{eq:can_2}kb_3+b_4 & = k(\enn-1)(u-b_1) + (k+1)(b_0+u-b_1).
  \end{align}
  Reducing \eqref{eq:can_2} modulo \(\enn-1\) yields
  \[b_4 = u-b_1+b_0 = b_3\mod \enn-1.\] We distinguish the following
  cases:
  \begin{enumerate}
  \item \(b_3=r\), \(b_4=q(\enn-1)+r\) with
    \(q=0,1,\ldots,k-1\) and \(r=0,1,\ldots,\enn-1\).
  \item \(b_3=\enn-1\), \(b_4=0\)
  \item \(b_3=0\), \(b_4=k(\enn-1)\)
  \item \(b_3=b_4=-1\).
  \end{enumerate}

  We illustrate Cases 1--3 in the diagram below for \(\enn=2\), \(k=3\):

  
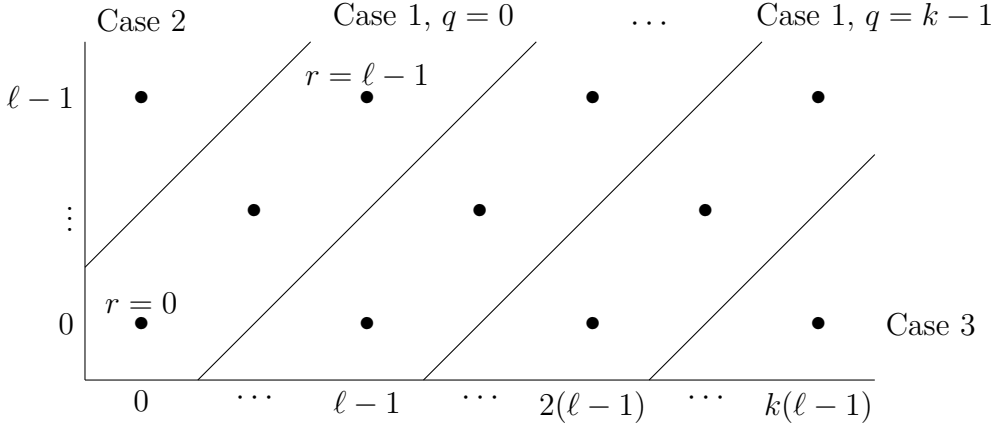
\begin{center}
    \begin{tikzpicture}
      \draw (0,0) -- (1.5*7,0);
      \draw (0,0) -- (0,1.5*3);
      \node at (0,1.5*0.5) [left] {\(0\)};
      \node at (0,1.5*1.5) [left] {\(\vdots\)};
      \node at (0,1.5*2.5) [left] {\(\enn-1\)};   
      \node at (1.5*0.5,0) [below] {\(0\)};
      \node at (1.5*1.5,0) [below] {\(\cdots\)};
      \node at (1.5*2.5,0) [below] {\(\enn-1\)};
      \node at (1.5*3.5,0) [below] {\(\cdots\)};
      \node at (1.5*4.5,0) [below] {\(2(\enn-1)\)};
      \node at (1.5*5.5,0) [below] {\(\cdots\)};
      \node at (1.5*6.5,0) [below] {\(k(\enn-1)\)};
      \node at (1.5*0.5,1.5*0.5) {\(\bullet\)};
      \node at (1.5*0.5,1.5*0.5) [above] {\(r=0\)};
      \node at (1.5*0.5,1.5*2.5) {\(\bullet\)};
      \node at (1.5*1.5,1.5*1.5) {\(\bullet\)};
      \node at (1.5*2.5,1.5*0.5) {\(\bullet\)};
      \node at (1.5*2.5,1.5*2.5) {\(\bullet\)};
      \node at (1.5*2.5,1.5*2.5) [above] {\(r=\enn-1\)};
      \node at (1.5*3.5,1.5*1.5) {\(\bullet\)};
      \node at (1.5*4.5,1.5*0.5) {\(\bullet\)};
      \node at (1.5*4.5,1.5*2.5) {\(\bullet\)};
      \node at (1.5*5.5,1.5*1.5) {\(\bullet\)};
      \node at (1.5*6.5,1.5*0.5) {\(\bullet\)};
      \node at (1.5*6.5,1.5*2.5) {\(\bullet\)};
      \node at (1.5*0.5,1.5*3) [above] {Case 2};
      \node at (1.5*7,1.5*0.5) [right] {Case 3};
      \draw (0,1.5) -- (1.5*2, 1.5*3);
      \draw (1.5*1,0) -- (1.5*4, 1.5*3);
      \draw (1.5*3,0) -- (1.5*6, 1.5*3);
      \draw (1.5*5,0) -- (1.5*7, 1.5*2);
      \node at (1.5*3,1.5*3) [above] {Case 1, \(q=0\)};
      \node at (1.5*5,1.5*3) [above] {\(\cdots\)};
      \node at (1.5*7,1.5*3) [above] {Case 1, \(q=k-1\)};
    \end{tikzpicture}
  \end{center}
  
  In what follows, we let
  \(d=\gcd(\enn-1,k+1)=\gcd(k+1,k\enn+1)=\gcd(\enn-1,k\enn+1)\) and define
  \(x,y,z\) by
  \[k+1=dx,\quad k\enn+1=dy,\quad \enn-1=dz.\] We will focus on type \(A\)
  contributions (there will be corresponding type \(B\) contributions
  obtained by multiplying with \(x_0 x_0^\vee\)).
  
  In Case 1, Equation \eqref{eq:can_2} becomes
  \[(\enn-1)q+(k+1)r=k(\enn-1)(u-b_1)+(k+1)(b_0+u-b_1),\] so
  \(k(u-b_1) = q - s x\), \(b_0+u-b_1=r + s z\) for some
  integer \(s\). Equation \eqref{eq:can_1} tells us that \(r + s z
  = r\mod \enn-1\), so \(s=dP\) for some integer \(P\). If we write
  \(P=kp+q\) for some \(p\), we get
  \[k(u-b_1) = q-sx = q-dPx= q-dxkp-dxq = -k(q+(k+1)p),\] giving
  \(u-b_1=-(q+(k+1)p)\), and
  \begin{align*}
    b_0 &= r+sz - (u-b_1)\\
        &= r+dz(kp+q)+q+(k+1)p\\
        & = r+kp(\enn-1)+q(\enn-1)+q+(k+1)p.\\
        &= q\enn+r+(k\enn+1)p.
  \end{align*}
  This yields a contribution of
  \[x_0^{(k\enn+1)p+q \enn+r} v x_3^r x_4^{q(\enn-1)+r}
    \in \HH^{-2((k+1)p+q)}\mbox{ where } v=\begin{cases}1&\mbox{
        if }b_0=0\mod 2, \\x_1^\vee x_2^\vee&\mbox{ if } b_0=1\mod
      2.\end{cases}\] In each case, there is precisely one \(\gamma\)
  contributing this monomial (according to whether the fixed variables
  are \(\{0,1,2,3,4\}\) or \(\{0,3,4\}\)). There are \(\enn\)
  contributions in each degree (as \(r\) varies) and we get every
  degree congruent to \(-2q\mod 2(k+1)\) for
  \(q=0,1,\ldots,k-1\), that is, \(\HH^d\) has rank \(\enn\) for
  every even \(d\neq -2k\mod 2(k+1)\), \(d\leq 0\). The corresponding
  type \(B\) contributions give \(\dim\HH^d=\enn\) for every odd
  \(d\neq 1-2k\mod 2(k+1)\), \(d\leq 1\).

  In Cases 2 and 3, \(kb_3+b_4=k(\enn-1)\), so
  \[k(\enn-1)=k(\enn-1)(u-b_1) + (k+1)(b_0+u-b_1),\] which implies
  \[k(u-b_1) = k - s x,\quad b_0 + u - b_1 = s z\]
  for some \(s\). As before, Equation \eqref{eq:can_1} implies
  \(s=dP\), so \(k(u-b_1) = k - P(k+1)\). This means \(P=kp\) for some
  \(p\), so \(u-b_1 = 1-(k+1)p\) and \(b_0=(k\enn+1)p-1\). Thus we get
  contributions
  \[x_0^{(k\enn+1)p-1}vw\in\HH^{-2(k+1)p+2}\mbox{ where }
    w\in\left\{x_3^{\enn-1},x_4^{k(\enn-1)}\right\}\mbox{ and }
    v=\begin{cases}1&\mbox{ if }b_0=0\mod 2, \\x_1^\vee
      x_2^\vee&\mbox{ if }b_0=1\mod 2.\end{cases}\] In both cases
  there is precisely one \(\gamma\) contributing this monomial
  (according to whether the fixed variables are \(\{0,1,2,3,4\}\) or
  \(\{0,3,4\}\)). This gives two contributions in every even degree
  \(d=2\mod 2(k+1)\), \(d\leq -2k\).
  
  Finally, in Case 4 we have \(kb_3+b_4=-(k+1)\), which yields
  \[u-b_1=-(k+1)p,\quad b_0=(k\enn+1)p-1,\] and we get a contribution
  \[x_0^{(k\enn+1)p-1}vx_3^{\vee}x_4^{\vee}\in\HH^{-2(k+1)p+2}\mbox{
      where } v=\begin{cases}1&\mbox{ if }b_0=0\mod 2, \\x_1^\vee
      x_2^\vee&\mbox{ if } b_0=1\mod 2.\end{cases}\] In both cases,
  there are \(\enn-2\) elements \(\gamma\) contributing these monomials
  (according to whether the fixed variables are \(\{0,1,2\}\) or
  \(\{0\}\)). Together with the contributions from Cases 2 and 3, this
  yields \(\dim\HH^d=\enn\) for every even \(d=2\mod 2(k+1)\),
  \(d\leq -2k\). The corresponding type \(B\) contributions give
  \(\dim\HH^d = \enn\) in every odd degree \(d=3\mod 2(k+1)\),
  \(d\leq 3\).

  Altogether, we get \(\dim\HH^d = \enn\) if \(d\leq 0\) and
  \(\dim\HH^3=(k\enn+1)(\enn-1)\).\qedhere
\end{proof}

\section{Bigrading}

\subsection{Scale-equivalence of bigradings}

In this section, we need to work over \(\C\) (or at least an
algebraically closed field of characteristic zero).

\begin{definition}
  A \(\Z\times\C\)-grading on a vector space \(V\) (or {\em bigrading}
  for short) is a
  decomposition \[V=\bigoplus_{(p,q)\in\Z\times\C}V^{p,q}.\] Two
  \(\Z\times\C\)-graded vector spaces \(V=\bigoplus V^{p,q}\) and
  \(W=\bigoplus W^{p,q}\) are {\em scale-equivalent} if there is a
  nonzero \(c\in\C\) such that \(\dim(V^{p,q})=\dim(W^{p,cq})\) for
  all \(p,q\).
\end{definition}

Our contact invariant will be a scale-equivalence class of
\(\Z\times\C\)-graded vector spaces (in fact, we will be able to find a representative which takes values in \(\Z \times \Z\)). We now explain how to construct a \(\Z\times\C\)-graded vector spaces out of a certain class of Gerstenhaber algebras.

\subsection{Bigradings from Gerstenhaber algebras}

Let \(\g^*\) be a Gerstenhaber algebra over \(\C\); in particular,
there is a Gerstenhaber bracket \([\cdot,\cdot]\) on \(\g^*\)
satisfying:
\[[x,y]= (-1)^{|x||y|}[y,x],\quad
(-1)^{|x||z|} [ [x,y],z] + (-1)^{|y||x|} [[y,z],x] + (-1)^{|z||y|} [[z,x],y] = 0 \]
The subset
\(\g^1\subset\g^*\) is a complex Lie algebra and the bracket gives a
representation \(\rho^d\colon\g^1\to\gl(\g^d)\) for each
\(d\). We will assume that each graded piece of \(\g^*\) is
finite-dimensional.

Let \(\h\subset\g^1\) be a Cartan subalgebra, that is a nilpotent,
self-normalising subalgebra. A Cartan subalgebra exists and is unique
up to automorphisms of \(\g^1\); for example, you can construct one by
taking the generalised 0-eigenspace of a regular element (an element
\(\zeta\in\g^1\) is regular if the generalised 0-eigenspace of
\(\OP{ad}_{\zeta}\) has the least possible dimension). If
\(\rho\colon\g^1\to\gl(V)\) is a finite-dimensional complex
representation then we get a weight-space decomposition
\(V = \bigoplus_{\alpha\in \h^*}V_\alpha\) where
\[V^\alpha := \left\{v\in V\ :\ (\rho(H) - \alpha(H))^Nv = 0\mbox{ for
      some }N\right\}.\] In other words, \(V^\alpha\) is a
simultaneous generalised eigenspace for \(\{\rho(H)\ :\ H\in\h\}\),
with eigenvalues \(\alpha(H)\). The weight-space decomposition
\(\g^1 = \bigoplus_\alpha\g^{1,\alpha}\) of the adjoint representation
has \(\h = \g^{1,0}\).

If \(\h\) has rank 1 then we have \(\h^*\cong\C\). If we pick such an
identification then the weight-space decomposition gives us a
\(\Z\times\C\)-bigrading \(\g^* = \bigoplus_{p,q}\g^{p,q}\). Changing
our identification \(\h^*\cong\C\) yields a scale-equivalent
\(\Z\times\C\)-grading.

\begin{exm}\label{exm:hh_formal}
  Let \(A^*\) be a \(\Z\)-graded associative algebra and suppose that
  its Hochschild cohomology \(\HH^{*}(A,A)\) has finite dimension in
  each degree. The Hochschild cochains can be given an additional
  \(\Z\)-grading so that a graded multilinear map
  \(A^{\otimes p}\to A[-q]\) contributes to
  \(\HH^{p,q}(A,A)\subset\HH^{p+q}(A,A)\). This
  \(\Z\times\Z\)-bigrading fits into our setting above. We write
  \(\HH^*(A,A)\cong\bigoplus_{p,q}\HH^{p,q}(A,A)\).  There is an
  element \(\eu\in\OP{CC}^{1,0}(A,A)\) defined on the graded piece
  \(A^q\) by \(\eu(a)=qa\). This is a Hochschild cocycle and defines a
  class (which we also write as \(\eu\)) in \(\HH^{1,0}(A,A)\). This
  satisfies \([\eu,c]=qc\) for \(c\in\OP{CC}^{*,q} (A,A)\). In
  particular, the generalised\footnote{Since \(\OP{ad}_{\eu}\) is
    semisimple on the level of cochains, it remains semisimple in its
    action on cohomology, so generalised eigenspaces are actual
    eigenspaces.}  0-eigenspace of \(\OP{ad}_{\eu}\) is
  \(\HH^{1,0}(A,A)\). If \(\eu\) is a regular element of the Lie
  algebra \(\HH^1(A,A)\) then \(\HH^{1,0}(A,A)\) is a Cartan
  subalgebra. In particular, if \(\dim(\HH^{1,0}(A,A))=1\) then
  \(\eu\) is necessarily regular and we can take
  \(\h=\HH^{1,0}(A,A)\). In this case, if we identify \(\h^*\) with
  \(\C\) by sending \(\eu^*\) to \(1\) then the weight decomposition
  gives us the usual bigrading
  \(\HH^*(A,A)\cong\bigoplus_{p,q}\HH^{p,q}(A,A)\).
\end{exm}

\subsection{Bigradings on symplectic cohomology}

If \(V\) is a Liouville domain with \(c_1(V)=0\), the symplectic
cohomology \(\SH^*(V)\) is a Gerstenhaber algebra. We will sketch how
the bracket is defined; for more detail, see {\cite[Section
  4]{SeidelDisj}} or {\cite[Section 2.5.1]{Abouzaid}}. The bracket
\([x,y]\) is defined
by \[[x,y]=\bigoplus_z\left(\sharp\MM(z;x,y,H,J)\right)z,\] where
\(\MM(z;x,y,J)\) is the moduli space of solutions
\(u\colon\Sigma\to \hat{V}\) to Floer's equation
\[\left(du + X_H\otimes\beta\right)^{0,1} = 0\] where:
\begin{itemize}
\item \(\hat{V}\) is the symplectic completion of \(V\);
\item \(\Sigma\) is a pair-of-pants
  \(\mathbb{CP}^1\setminus\{0,1,\infty\}\), where we consider \(0,1\)
  to be positive punctures and \(\infty\) as a negative puncture;
\item we equip \(\Sigma\) with a 1-parameter family of
  positive/negative cylindrical ends, specified by asymptotic markers
  which rotate once for each puncture. As the parameter varies from
  \(0\) to \(2\pi\), the markers at the positive punctures rotate once
  clockwise and the marker at the negative puncture rotates once
  anticlockwise;
\item \(\beta\) is a subclosed 1-form on \(\Sigma\) compatible with
  the cylindrical ends;
\item \(u\) has asymptotes \(x\), \(y\), \(z\) respectively at the
  punctures \(0\), \(1\), \(\infty\).
\end{itemize}
The bracket has degree \(-1\), that is
\[|x|+|y|=|z|+1,\] where the degree is related to the Conley-Zehnder
index by \(|x|=n-\mu_{CZ}(x)\). Equivalently,
\[n=\mu_{CZ}(x)+\mu_{CZ}(y)-\mu_{CZ}(z)+1.\]

\begin{lem}\label{lma:stretch}
  Let \(V\) be a \(2n\)-dimensional Liouville domain with
  simply-connected boundary and suppose that there is a contact form
  on \(Y=\partial V\) such that every Reeb orbit \(\gamma\) on \(Y\)
  satisfies the inequality
  \[\mu_{CZ}(\gamma) \geq \max(5 - n, n - 1).\] If \(x,y,z\) are Reeb
  orbits then there
  exists a \(J\) such that any \(u\in\MM(z;x,y,H,J)\) avoids the
  interior of \(V\), that is, every \(u\) stays in the cylindrical end
  \(\hat{V}\setminus V\).
\end{lem}
\begin{proof}
  Suppose this is not true. Pick a neck-stretching sequence of
  almost complex structures \(J_k\) around \(Y\) and assume our
  Hamiltonian is constant on the neck as in {\cite[Figure
    8]{CieliebakOancea}} so that our solutions to Floer's
  equation are genuinely holomorphic in that region and the
  standard SFT analysis of neck-stretching applies. Suppose we
  have a sequence of curves \(u_k\in\MM(z;x,y,H,J_k)\) which
  enter the interior of \(V\). By the SFT compactness theorem,
  we can find a convergent subsequence which breaks into
  levels. There are several cases we need to consider.

  Case 1: A break occurs along a separating curve parallel to \(z\) (and
  possibly other curves).

  
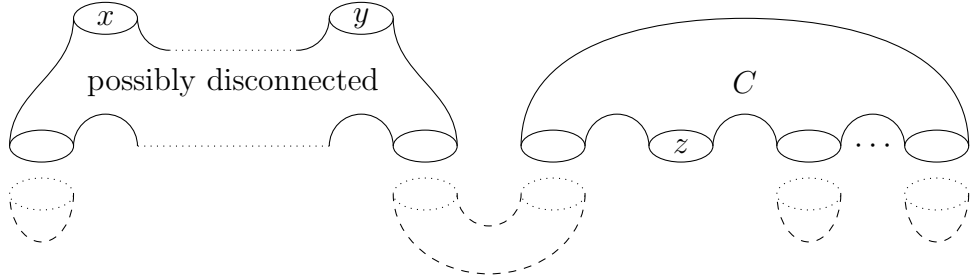
\begin{center}
    \begin{tikzpicture}
      \begin{scope}[scale=0.85]
        \draw (0,0) circle [x radius=0.5, y radius=0.25];
        \draw (4,0) circle [x radius=0.5, y radius=0.25];
        \draw (-1,-2) circle [x radius=0.5, y radius=0.25];
        \draw (5,-2) circle [x radius=0.5, y radius=0.25];
        \node at (0,0) {\(x\)};
        \node at (4,0) {\(y\)};
        \draw (-0.5,0) to[out=-90,in=90] (-1.5,-2);
        \draw (0.5,0) to[out=-90,in=180] (1,-0.5);
        \draw (-0.5,-2) to[out=90,in=180] (-0,-1.5) to[out=0,in=90] (0.5,-2);
        \draw (3.5,0) to[out=-90,in=0] (3,-0.5);
        \draw[dotted] (3,-0.5) -- (1,-0.5);
        \draw[dotted] (3.5,-2) -- (0.5,-2);
        \node at (2,-1) {possibly disconnected};
        \draw (3.5,-2) to[out=90,in=180] (4,-1.5) to[out=0,in=90] (4.5,-2);
        \draw (4.5,0) to[out=-90,in=90] (5.5,-2);
        \disccap{-1}{-2.75}
        \cylind{5}{-2.75}
        \begin{scope}[shift={(10,0)}]
          \draw (-3,-2) circle [x radius=0.5, y radius=0.25];
          \draw (-1,-2) circle [x radius=0.5, y radius=0.25];
          \draw (1,-2) circle [x radius=0.5, y radius=0.25];
          \draw (3,-2) circle [x radius=0.5, y radius=0.25];
          \node at (-1,-2) {\(z\)};
          \draw (-3.5,-2) to[out=90,in=180] (0,0) to[out=0,in=90] (3.5,-2);
          \draw (-2.5,-2) to[out=90,in=180] (-2,-1.5) to[out=0,in=90] (-1.5,-2);
          \draw (-0.5,-2) to[out=90,in=180] (0,-1.5) to[out=0,in=90] (0.5,-2);
          \draw (1.5,-2) to[out=90,in=180] (2,-1.5) to[out=0,in=90] (2.5,-2);
          \node at (2,-2) {\(\cdots\)};
          \disccap{1}{-2.75}
          \disccap{3}{-2.75}
          \node at (0,-1) {\(C\)};
        \end{scope};
      \end{scope};
    \end{tikzpicture}
  \end{center}
  
  Case 2: Not case 1, but a break occurs along separating curves
  parallel to \(x\) and to \(y\).

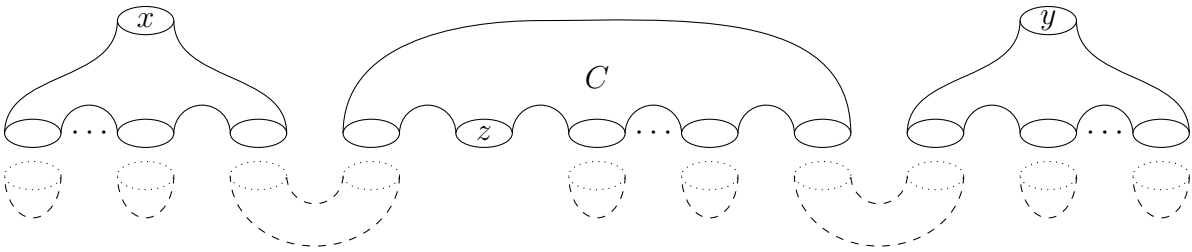
  \begin{center}
    \begin{tikzpicture}
      \begin{scope}[scale=0.75]
        \draw (-3,-2) circle [x radius=0.5, y radius=0.25];
        \draw (-1,-2) circle [x radius=0.5, y radius=0.25];
        \draw (1,-2) circle [x radius=0.5, y radius=0.25];
        \draw (3,-2) circle [x radius=0.5, y radius=0.25];
        \draw (5,-2) circle [x radius=0.5, y radius=0.25];
        \node at (-1,-2) {\(z\)};
        \draw (-3.5,-2) to[out=90,in=180] (0,0) to[out=0,in=90] (5.5,-2);
        \draw (-2.5,-2) to[out=90,in=180] (-2,-1.5) to[out=0,in=90] (-1.5,-2);
        \draw (-0.5,-2) to[out=90,in=180] (0,-1.5) to[out=0,in=90] (0.5,-2);
        \draw (1.5,-2) to[out=90,in=180] (2,-1.5) to[out=0,in=90] (2.5,-2);
        \draw (3.5,-2) to[out=90,in=180] (4,-1.5) to[out=0,in=90] (4.5,-2);
        \node at (2,-2) {\(\cdots\)};
        \disccap{1}{-2.75}
        \disccap{3}{-2.75}
        \node at (1,-1) {\(C\)};
        \cylind{-5}{-2.75}
        \cylind{5}{-2.75}
        \begin{scope}[shift={(-7,0)}]
          \draw (0,0) circle [x radius=0.5, y radius=0.25];
          \draw (-2,-2) circle [x radius=0.5, y radius=0.25];
          \draw (0,-2) circle [x radius=0.5, y radius=0.25];
          \draw (2,-2) circle [x radius=0.5, y radius=0.25];
          \node at (0,0) {\(x\)};
          \draw (-0.5,0) to[out=-90,in=90] (-2.5,-2);
          \draw (0.5,0) to[out=-90,in=90] (2.5,-2);
          \draw (-1.5,-2) to[out=90,in=180] (-1,-1.5) to[out=0,in=90] (-0.5,-2);
          \draw (0.5,-2) to[out=90,in=180] (1,-1.5) to[out=0,in=90] (1.5,-2);
          \node at (-1,-2) {\(\cdots\)};
          \disccap{-2}{-2.75}
          \disccap{0}{-2.75}
        \end{scope};
        \begin{scope}[shift={(9,0)}]
          \draw (0,0) circle [x radius=0.5, y radius=0.25];
          \draw (-2,-2) circle [x radius=0.5, y radius=0.25];
          \draw (0,-2) circle [x radius=0.5, y radius=0.25];
          \draw (2,-2) circle [x radius=0.5, y radius=0.25];
          \node at (0,0) {\(y\)};
          \draw (-0.5,0) to[out=-90,in=90] (-2.5,-2);
          \draw (0.5,0) to[out=-90,in=90] (2.5,-2);
          \draw (-1.5,-2) to[out=90,in=180] (-1,-1.5) to[out=0,in=90] (-0.5,-2);
          \draw (0.5,-2) to[out=90,in=180] (1,-1.5) to[out=0,in=90] (1.5,-2);
          \node at (1,-2) {\(\cdots\)};
          \disccap{2}{-2.75}
          \disccap{0}{-2.75}
        \end{scope};
      \end{scope};
    \end{tikzpicture}
  \end{center}
  
  Case 3\(_x\): Not cases 1--2, but a break occurs along a separating
  curve parallel to \(x\) (Case 3\(_y\) similar).

  
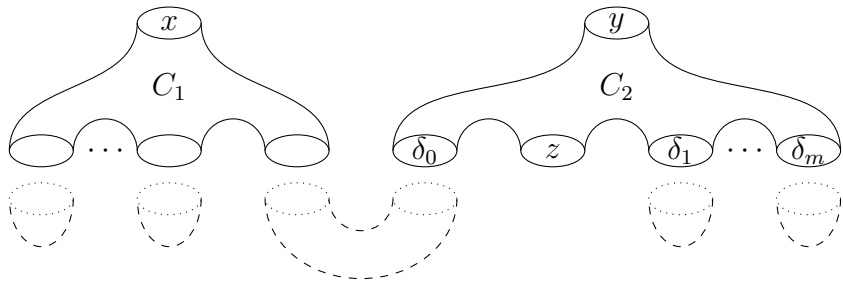
\begin{center}  
    \begin{tikzpicture}
      \begin{scope}[scale=0.85]
        \node at (0,-1) {\(C_1\)};
        \node at (7,-1) {\(C_2\)};
        \draw (0,0) circle [x radius=0.5, y radius=0.25];
        \draw (-2,-2) circle [x radius=0.5, y radius=0.25];
        \draw (0,-2) circle [x radius=0.5, y radius=0.25];
        \draw (2,-2) circle [x radius=0.5, y radius=0.25];
        \node at (0,0) {\(x\)};
        \draw (-0.5,0) to[out=-90,in=90] (-2.5,-2);
        \draw (0.5,0) to[out=-90,in=90] (2.5,-2);
        \draw (-1.5,-2) to[out=90,in=180] (-1,-1.5) to[out=0,in=90] (-0.5,-2);
        \draw (0.5,-2) to[out=90,in=180] (1,-1.5) to[out=0,in=90] (1.5,-2);
        \node at (-1,-2) {\(\cdots\)};
        \begin{scope}[shift={(7,0)}]
          \draw (0,0) circle [x radius=0.5, y radius=0.25];
          \draw (-3,-2) circle [x radius=0.5, y radius=0.25];
          \draw (-1,-2) circle [x radius=0.5, y radius=0.25];
          \draw (1,-2) circle [x radius=0.5, y radius=0.25];
          \draw (3,-2) circle [x radius=0.5, y radius=0.25];
          \node at (0,0) {\(y\)};
          \node at (-3,-2) {\(\delta_0\)};
          \node at (-1,-2) {\(z\)};
          \node at (1,-2) {\(\delta_1\)};
          \node at (3,-2) {\(\delta_m\)};
          \draw (-0.5,0) to[out=-90,in=90] (-3.5,-2);
          \draw (0.5,0) to[out=-90,in=90] (3.5,-2);
          \draw (-2.5,-2) to[out=90,in=180] (-2,-1.5) to[out=0,in=90] (-1.5,-2);
          \draw (-0.5,-2) to[out=90,in=180] (-0,-1.5) to[out=0,in=90] (0.5,-2);
          \draw (1.5,-2) to[out=90,in=180] (2,-1.5) to[out=0,in=90] (2.5,-2);
          \node at (2,-2) {\(\cdots\)};
        \end{scope};
        \disccap{-2}{-2.75}
        \disccap{0}{-2.75}
        \disccap{8}{-2.75}
        \disccap{10}{-2.75}
        \cylind{2}{-2.75}
      \end{scope}
    \end{tikzpicture}
  \end{center}
  
  Case 4: Not cases 1--3, but a break occurs along a contractible loop.
  
  
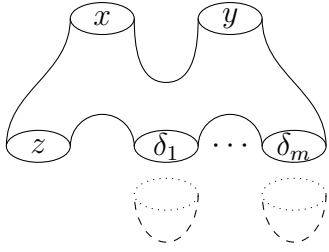
\begin{center}
    \begin{tikzpicture}
      \begin{scope}[scale=0.85]
        \draw (-1,0) circle [x radius=0.5, y radius=0.25];
        \draw (1,0) circle [x radius=0.5, y radius=0.25];
        \draw (-2,-2) circle [x radius=0.5, y radius=0.25];
        \draw (-0,-2) circle [x radius=0.5, y radius=0.25];
        \draw (2,-2) circle [x radius=0.5, y radius=0.25];
        \node at (-1,0) {\(x\)};
        \node at (1,0) {\(y\)};
        \node at (-2,-2) {\(z\)};
        \node at (0,-2) {\(\delta_1\)};
        \node at (2,-2) {\(\delta_m\)};
        \node at (1,-2) {\(\cdots\)};
        \draw (-1.5,-2) to[out=90,in=180] (-1,-1.5) to[out=0,in=90] (-0.5,-2);
        \draw (0.5,-2) to[out=90,in=180] (1,-1.5) to[out=0,in=90] (1.5,-2);
        \draw (-1.5,0) to[out=-90,in=90] (-2.5,-2);
        \draw (1.5,0) to[out=-90,in=90] (2.5,-2);
        \draw (-0.5,0) to[out=-90,in=180] (0,-1) to[out=0,in=-90] (0.5,0);
        \disccap{2}{-2.75}
        \disccap{0}{-2.75}
      \end{scope}
    \end{tikzpicture}
  \end{center}
  
  In Cases 1--2, we are left with a component \(C\) which violates the
  maximum principle (see also the argument from Bourgeois-Oancea
  {\cite[Proof of Proposition 5, Step 1]{bouoan}} or an alternative
  argument based on action from Cieliebak-Oancea {\cite[Proof of
    Proposition 9.17]{CieliebakOancea}}).
  
  The argument for Case 3\(_x\) is inspired by {\cite[Appendix
    A]{CieliebakOancea}} and {\cite[Lemma 3.13]{UebelePeriodic}}. In
  this case, there are at least two components \(C_1\) and \(C_2\) in
  the SFT limit, where \(C_1\) has \(x\) as a positive asymptote and
  \(C_2\) has \(y\) as a positive asymptote. The component \(C_2\) has
  a negative asymptote at \(z\), a negative asymptote \(\delta_0\)
  which connects through lower levels to the component \(C_1\), and
  possibly further negative asymptotes \(\delta_1,\ldots,\delta_m\),
  which are capped off by planes in other levels. The index of \(C_2\)
  is (we justify this in Remark \ref{rmk:index_formula} below):
  \begin{equation}
    \label{eq:index}\mu_{CZ}(y) - \mu_{CZ}(z) + 1 - \sum_{i=0}^m(\mu_{CZ}(\delta_i)+n-3).
  \end{equation}
  We have
  \[\mu_{CZ}(y)-\mu_{CZ}(z)+1 = n-\mu_{CZ}(x)\leq 1\] because \(\mu_{CZ}(x)\geq n-1\) by
  assumption. Moreover \(\mu_{CZ}(\delta_i)+n-3\geq 2\) by assumption, so
  \(\mu_{CZ}(y)-\mu_{CZ}(z)+1-\sum_{i=0}^m\mu_{CZ}(\delta_i)\leq 1-2=-1\), which
  contradicts the regularity of \(C_2\).

  The argument for Case 3\(_y\) is the same as for Case 3\(_x\) with
  the roles of \(x\) and \(y\) interchanged.

  Case 4 yields a regular component \(C\) in the SFT limit which has
  punctures asymptotic to \(x\), \(y\) and \(z\) as well as further
  negative punctures with asymptotes \(\delta_1,\ldots,\delta_m\). The
  index of \(C\) is equal to the index of the original moduli space
  minus \(\sum_{i=1}^m(\mu_{CZ}(\delta_i)+n-3)\geq 2\), so becomes
  negative. This is a contradiction.
\end{proof}

\begin{rmk}\label{rmk:index_formula}
  We now explain the index formula \eqref{eq:index} from the proof. If
  we fix the positions of the punctures and all the asymptotic
  markers, the virtual dimension of this moduli space is (see
  Schwarz's thesis {\cite[Theorem 3.3.11]{Schwarz}}):
  \[\mu_{CZ}(y) - \mu_{CZ}(z) - \sum_{i=0}^m \mu_{CZ}(\delta_i) - n(m+1)\] since
  \(-m-1\) is the Euler characteristic of the domain. However, the
  bubbling/breaking which gives rise to the punctures at \(\delta_i\)
  can happen anywhere, with any asymptotic marker, and the asymptotic
  markers on \(y\) and \(z\) can move in a 1-parameter family, so we
  get an additional \(3(m+1)+1\), which gives Equation
  \eqref{eq:index}. Note that this is intermediate between the formula
  in Schwarz's thesis and the formula {\cite[Corollary
    5.4]{Bourgeois}} from Bourgeois's thesis, where all punctures and
  markers are allowed to move.
\end{rmk}

\begin{cor}\label{cor:lie_alg}
  Suppose that \((Y,\xi)\) is a \((2n-1)\)-dimensional contact
  manifold which admits a contact form \(\alpha\) for which every
  closed Reeb orbit \(\gamma\) satisfies
  \[\mu_{CZ}(\gamma) \geq \max(5 - n, n - 1).\] Let \(V_1\), \(V_2\) be
  Liouville domains with \(c_1(V_i) = 0\) and \(\partial V_i =
  Y\). Suppose that \(V_i\) admits a Morse function with no critical
  points of index \(1\). Then (a) there is an isomorphism of Lie
  algebras \(f^1\colon \SH^1(V_1)\to \SH^1(V_2)\), and (b) for each
  \(d<0\) there is an isomorphism
  \(f^d\colon \SH^d(V_1)\to \SH^d(V_2)\) which intertwines the
  representations
  \[\OP{ad}\colon \SH^1(V_i)\to \mathfrak{gl}(\SH^d(V_i)).\] That is,
  for each \(d<0\), we have a commutative diagram:

  \begin{center}
    \begin{tikzpicture}
      \node (A) at (0,0) {\(\SH^1(V_1)\)};
      \node (B) at (0,-1.5) {\(\SH^1(V_2)\)};
      \node (C) at (4,0) {\(\mathfrak{gl}(\SH^d(V_1))\)};
      \node (D) at (4,-1.5) {\(\mathfrak{gl}(\SH^d(V_2))\)};
      \draw[->] (A) -- (C) node [midway,above] {\(\OP{ad}\)};
      \draw[->] (B) -- (D) node [midway,below] {\(\OP{ad}\)};
      \draw[->] (A) -- (B) node [midway,left] {\(f^1\)};
      \draw[->] (C) -- (D) node [midway,right] {\(f^d\)};
    \end{tikzpicture}
  \end{center}
\end{cor}
\begin{proof}
  Under the assumptions of the corollary, every element of
  \(\SH^1(V_i)\) or of \(\SH^d(V_i)\) with \(d<0\) can be represented
  using Reeb orbits for the contact form \(\alpha\) (rather than
  critical points of a Morse function on the filling). These Reeb
  orbits lie in the cylindrical end of the symplectic completion
  \(\hat{V}_i\) (rather than in the filling), and these cylindrical
  ends are both symplectomorphic to the half-symplectisation
  \([0,\infty)\times Y\), so in a suitable cochain model of symplectic
  cohomology, we get identifications
  \(f^d\colon SC^d(V_1)\to SC^d(V_2)\) when \(d=1\) or \(d<0\). Since
  the contact boundary is index-positive, we know that the
  differential on these cochain groups is independent of the filling
  {\cite[Prop. 9.17]{CieliebakOancea}}, which tells us that \(f^1\)
  and \(f^d\) induce isomorphisms on cohomology.

  By Lemma \ref{lma:stretch}, we know there exist almost complex
  structures for which the Gerstenhaber bracket between these orbits
  does not involve any contributions from curves entering the
  filling. This implies that \(f^1\) is an isomorphism of Lie algebras
  and that \(f^d\) intertwines the adjoint action of \(\SH^1\).  
\end{proof}

\subsection{Bigrading on \(\HH^*(\mf)\)}
\label{sct:bigrading_mf}

In all our examples, we calculated \(\HH^*(\mf(\A^{n+1},\Gamma_\w,\w))\)
and saw that \(\HH^2 = 0\). Moreover, we saw in Lemma \ref{lma:formal}
that there is an intrinsically formal algebra \(B\) such that
\[\HH^*(\mf(\A^{n+1},\Gamma_\w,\w)) = \HH^*(B).\]
We now compute the usual algebra bigrading on \(\HH^*(B)\) in terms of
the \(\gamma\)-monomial contributions from Theorem
\ref{thm:hh_formula}.

\begin{lem}\label{lma:x_0_power}
  A \(\gamma\)-monomial \(m\) contributing to \(\HH^d(B)\) contributes
  to the bigraded piece \(\HH^{d-nb_0,nb_0}(B)\), where \(b_0\) is
  the total exponent of \(x_0\) in \(m\).
\end{lem}
\begin{rmk}
  Recall that \(\HH^d=\bigoplus_q\HH^{d-q,q}\), so this is really
  just saying that the bigrading of \(m\) is \( n b_0\).
\end{rmk}
\begin{proof}
  Consider the $\G_m$-action
        $t \cdot (x_0,\ldots, x_{n+1}) = (tx_0, \ldots, x_{n+1})$. Since this action
  leaves $\w$ invariant, its weights give a second grading on
  $\HH^*(\mf(\A^{n+2},\Gamma_\w,\w))$. Theorem \ref{thm:hh_formula}
  comes from an isomorphism between
  \(\HH^*(\mf(\A^{n+2},\Gamma_\w,\w))\) and a suitable twisted Koszul
  cohomology group (whose generators are \(\gamma\)-monomials)
  \cite{BFK}. This isomorphism respects the $\mathbb{G}_m$-action,
  hence this additional grading is given by the total exponent of
  \(x_0\) in the \(\gamma\)-monomials contributing to \(\HH\).

  As in Section \ref{sct:mirror}, let \(\SS = \bigoplus_i S_i\) be the
  generator of \(\mathcal{F}(V)\) given by a direct sum of vanishing
  cycles. In \cite[Theorem 4.2]{LU1} it is shown that, we have a
  generator $\EuScript{S}$ of $\mf(\A^{n+2}, \Gamma_\w \w)$ mirror to
  \(\SS\), where \(\EuScript{S}\) is the pushforward of a generator
  $\EuScript{E}$ of $\mf(\A^{n+1} , \Gamma_\w,\w)$ under the inclusion
  \((x_1,\ldots,x_{n+1})\mapsto (0,x_1,\ldots,x_{n+1})\). In
  particular, $\EuScript{S}$ is $\mathbb{G}_m$-invariant. Using this,
  \cite[Theorem 4.2]{LU1} shows that the endomorphism
  \(A_\infty\)-algebra $\BB=\mathrm{end}(\EuScript{S})$ is a formal
  algebra and the grading on $B = \mathrm{H}(\BB)$ is $n$ times the
  weight of the $\G_m$-action. Therefore, the $\G_m$-weight on $B$ can
  be understood in terms of the grading of the algebra $B$. Indeed, we
  see it as the weight decomposition for $\mathrm{ad}_b$ associated to
  the derivation \(b = n \cdot \eu\), where $\eu \in \HH^1$ is the
  Euler derivation from Example \ref{exm:hh_formal}.
\end{proof}

In fact, in all of our examples we have \(\dim\HH^{1,0}=1\), which
means, as in Example \ref{exm:hh_formal}, that the weight
decomposition for the representation
\(\OP{ad}\colon\HH^1 \to\bigoplus_d\mathfrak{gl}(\HH^d)\) gives a
\(\Z\times\C\)-bigrading which is scale-equivalent to the algebra
bigrading, and hence to the bigrading by the total exponent of \(x_0\)
by Lemma \ref{lma:x_0_power}.

\subsection{Proof of Theorem \ref{thm:distinguish}}

Let \(X\) be a cDV singularity and let \(V\) be its Milnor
fibre. Let \(\mu\) be the Milnor number of \(X\). By
{\cite[Theorem 6.6]{Milnor}}, the Milnor fibre admits a Morse
function with precisely one minimum and \(\mu\) critical points
of index \(3\); in particular, none of index 1. Since \(X\) is
terminal, McLean's theorem {\cite[Theorem 1.1]{McLean}} tells us
that there exists a contact form for which every closed Reeb
orbit \(\gamma\) satisfies
\(\mu_{CZ}(\gamma)\geq 2\OP{md}(X)=2\), where \(\OP{md}(X)\) is
the minimal discrepancy of \(X\), which equals 1 by a theorem of
Markushevich \cite{Markushevich}. Since \(n=3\), we have
\(\max(5-n,n-1)=2\), so that all the assumptions of Corollary
\ref{cor:lie_alg} are satisfied.

Consider the Lie algebra \(\SH^1(V)\) and its representation
\(\bigoplus_{d<0}\SH^d(V)\) (where \(\SH^1(V)\) acts by the
Gerstenhaber bracket). By Corollary \ref{cor:lie_alg}, this Lie
algebra representation is a contact invariant of the link of \(X\).

We know that Conjecture \ref{conj:BH} holds for all our Brieskorn-Pham
and \(cA_n\) examples, and we are going to assume that it holds for
the Laufer examples too. By Theorem \ref{thm:conditional}, this tells
us that, if \(V\) is the Milnor fibre of \(\check{\w}\), then
\[\SH^*(V)\cong \HH^*(\mf(\A^{n+2},\Gamma_\w,\w))\] as Gerstenhaber
algebras. Therefore, the contact invariant Lie algebra representation
is equivalent to the representation
\(\OP{ad}\colon\HH^1(B)\to\bigoplus_{d<0}\mathfrak{gl}(\HH^d(B))\)
discussed in Section \ref{sct:bigrading_mf}. In particular, this gives
a \(\Z\times\C\)-grading on \(\bigoplus_{d<0} \SH^d(V)\) which we can
compute in terms of the \(x_0\)-powers of the contributing
\(\gamma\)-monomials by Lemma \ref{lma:x_0_power}.

We now show that, for all of our examples, these scale-equivalence
classes of \(\Z\times\C\)-gradings distinguish the contact structures.

\subsubsection{\(\enn=1\)}

In this case we need to distinguish the contact structures
\(\{\alpha_{1,k}\ :\ k=1,2,3,\ldots\}\) and
\(\{\lambda_{1,k}\ :\ k=1,2,3,\ldots\}\) on \(S^2\times S^3\).

The unique contribution to \(\HH^{-2}\) is:
\[\begin{cases}x_0x_1^\vee x_2^\vee x_3^\vee x_4^\vee&\mbox{ for
    }\alpha_{1,1},\\ x_0^2x_4^2&\mbox{ for }\alpha_{1,k},\ (k\geq
    2),\end{cases}\qquad x_0^4x_1x_2^3\mbox{ for } \lambda_{1,k}.\] To
compare the \(\Z\times\C\)-gradings, we rescale to ensure
\(\SH^{-2,4}\neq 0\) in all cases. The \(\C\)-bigrading of a monomial
\(x_0^{b_0}\cdots x_4^{b_4}\in\HH^d\) is therefore given by:
\[\begin{cases}4b_0&\mbox{ for }\alpha_{1,1},\\ 2b_0&\mbox{ for }\alpha_{1,k},
  (k\geq 2),\end{cases}\qquad b_0\mbox{ for }\lambda_{1,k}.\] The unique
contribution to \(\HH^{-4}\) is:
\[\begin{cases}x_0^4\in \SH^{-4,16}&\mbox{ for
    }\alpha_{1,1},\\ x_0^3x_1^\vee x_2^\vee x_3^\vee x_4^\vee\in
    \SH^{-4,6}&\mbox{ for } \alpha_{1,2},\\ x_0^4x_4^4\in
    \SH^{-4,8}&\mbox{ for }\alpha_{1,k}, (k\geq 3),\end{cases} \qquad
  x_0^6x_2^4\in \SH^{-4,6}\mbox{ for }\lambda_{1,k}.\] This already
distinguishes \(\alpha_{1,1}\) from everything, \(\alpha_{1,2}\) from
the other \(\alpha\)s, and the \(\lambda\)s from the \(\alpha_{1,k}\),
\(k\neq 2\).

To distinguish \(\lambda_{1,k}\) from \(\lambda_{1,K}\) with
\(k < K\), observe that the unique contribution to \(\SH^{-4k-2}\) is
\(x_0^{6k+2}x_1^2\in \SH^{-4k-2,6k+2}\) respectively
\(x_0^{6k+4}x_1x_2^{4k+3}\in \SH^{-4k-2,6k+4}\).

To distinguish \(\alpha_{1,k}\) from \(\alpha_{1,K}\) with
\(2\leq k < K\), observe that the unique contribution to \(\SH^{-2k}\)
is \(x_0^{2k-1}x_1^\vee x_2^\vee x_3^\vee x_4^\vee\in \SH^{-2k,4k-2}\)
respectively \(x_0^{2k}x_4^{2k}\in \SH^{-2k,4k}\).

To distinguish \(\alpha_{1,2}\) from \(\lambda_{1,k}\), \(k\geq 2\),
observe that the unique contribution to \(\SH^{-6}\) is
\(x_0^4\in \SH^{-6,8}\) respectively
\(x_0^{10}x_1x_2^2\in \SH^{-6,10}\).

To distinguish \(\alpha_{1,2}\) from \(\lambda_{1,1}\), observe that
the unique contribution to \(\SH^{-8}\) is \(x_0^6x_2^2\in \SH^{-8,12}\)
respectively \(x_0^{10}\in \SH^{-8,10}\).

\subsection{\(\enn\geq 2\)}

The contact structures \(\xi_{\enn,k}\) in Theorem \ref{thm:distinguish}
live on the manifold \(\sharp_\ell(S^2\times S^3)\). We can see from the
tables in Theorems \ref{thm:bp}, \ref{thm:laufer} and \ref{thm:can}
that for any of \(\alpha_{\enn,k}\), \(\beta_{\enn,k}\), \(\delta_{4,k}\),
\(\lambda_{\enn,k}\), \(\epsilon_{6,k}\), \(\epsilon_{8,k}\), the
symplectic cohomology \(\SH^d\), \(d<0\), is supported in a single
\(\C\)-bigrading if and only if \(d=-2k\mbox{ or }1-2k\mod
2(k+1)\). Therefore, the only possibility for two contact structures
\(\xi_{\enn,k}\), \(\enn\geq 2\), to agree is for the indices \(\enn\) and
\(k\) to agree.

We also see that, if we bigrade by the total exponent of \(x_0\),
\(\SH^1\) is supported in bidegrees \(0,1,2,\ldots,\enn-1\). This is
enough to fix our \(\Z\times\C\)-grading completely up to scale so
that, in all cases, the \(\C\)-bigrading coincides with the total
exponent of \(x_0\).

To distinguish \(\alpha_{\enn,k}\) from \(\beta_{\enn,k}\) when \(k\neq 1\)
(the singularities are locally analytically equivalent when \(\enn=1\)),
observe that the contributions to \(\SH^{-2k}\) have total
\(x_0\) exponent \(k\enn\) respectively \(k(\enn+1)-1\). These are different
if \(k\neq 1\).

To distinguish \(\alpha_{4,k}\) and \(\beta_{4,k}\) from
\(\delta_{4,k}\), note that the contributions to \(\SH^{-2k}\) have
total \(x_0\) exponents respectively equal to \(5k-1\), \(4k\) and
\(6k-1\).

To distinguish \(\alpha_{6,k}\) and \(\beta_{6,k}\) from
\(\epsilon_{6,k}\), note that the contributions to \(\SH^{-2k}\) have
total \(x_0\) exponents respectively equal to \(7k-1\), \(6k\) and
\(12k-1\).

To distinguish \(\alpha_{8,k}\) and \(\beta_{8,k}\) from
\(\epsilon_{8,k}\), note that the contributions to \(\SH^{-2k}\) have
total \(x_0\) exponents respectively equal to \(9k-1\), \(8k\) and
\(30k-1\).

\bibliographystyle{plain}
\bibliography{cdv}

\begin{thebibliography}{10}

\bibitem{Abouzaid}
M.~Abouzaid.
\newblock Symplectic cohomology and {V}iterbo's theorem.
\newblock In {\em Free loop spaces in geometry and topology}, volume~24 of {\em
  IRMA Lect. Math. Theor. Phys.}, pages 271--485. Eur. Math. Soc., Z\"{u}rich,
  2015.

\bibitem{Arnold}
V.~I. Arnold, S.~M. Gusein-Zade, and A.~N. Varchenko.
\newblock {\em Singularities of differentiable maps. {V}olume 1}.
\newblock Modern Birkh\"{a}user Classics. Birkh\"{a}user/Springer, New York,
  2012.
\newblock Classification of critical points, caustics and wave fronts,
  Translated from the Russian by Ian Porteous based on a previous translation
  by Mark Reynolds, Reprint of the 1985 edition.

\bibitem{BFK}
M.~Ballard, D.~Favero, and L.~Katzarkov.
\newblock A category of kernels for equivariant factorizations and its
  implications for {H}odge theory.
\newblock {\em Publ. Math. Inst. Hautes \'{E}tudes Sci.}, 120:1--111, 2014.

\bibitem{BH}
P.~Berglund and T.~H\"{u}bsch.
\newblock A generalized construction of mirror manifolds.
\newblock In {\em Essays on mirror manifolds}, pages 388--407. Int. Press, Hong
  Kong, 1992.

\bibitem{Bourgeois}
F.~Bourgeois.
\newblock {\em A {M}orse-{B}ott approach to contact homology}.
\newblock ProQuest LLC, Ann Arbor, MI, 2002.
\newblock Thesis (Ph.D.)--Stanford University.

\bibitem{BEE}
F.~Bourgeois, T.~Ekholm, and Y.~Eliashberg.
\newblock Effect of {L}egendrian surgery.
\newblock {\em Geom. Topol.}, 16(1):301--389, 2012.
\newblock With an appendix by Sheel Ganatra and Maksim Maydanskiy.

\bibitem{bouoan}
F.~Bourgeois and A.~Oancea.
\newblock An exact sequence for contact- and symplectic homology.
\newblock {\em Invent. Math.}, 175(3):611--680, 2009.

\bibitem{BrieskornLink2}
E.~Brieskorn.
\newblock Beispiele zur {D}ifferentialtopologie von {S}ingularit\"{a}ten.
\newblock {\em Invent. Math.}, 2:1--14, 1966.

\bibitem{Brieskorn}
E.~Brieskorn.
\newblock Die {A}ufl\"{o}sung der rationalen {S}ingularit\"{a}ten holomorpher
  {A}bbildungen.
\newblock {\em Math. Ann.}, 178:255--270, 1968.

\bibitem{ghiggini}
B.~Chantraine, G.~Dimitroglou~Rizell, P.~Ghiggini, and R.~Golovko.
\newblock Geometric generation of the wrapped {F}ukaya category of {W}einstein
  manifolds and sectors.
\newblock {\em arXiv:1712.09126 (to appear in Ann. Sci. \'{E}cole Norm. Sup.)},
  2017.

\bibitem{CFH}
K.~Cieliebak, A.~Floer, and H.~Hofer.
\newblock Symplectic homology. {II}. {A} general construction.
\newblock {\em Math. Z.}, 218(1):103--122, 1995.

\bibitem{CieliebakOancea}
K.~Cieliebak and A.~Oancea.
\newblock Symplectic homology and the {E}ilenberg-{S}teenrod axioms.
\newblock {\em Algebr. Geom. Topol.}, 18(4):1953--2130, 2018.
\newblock Appendix written jointly with Peter Albers.

\bibitem{EL}
T.~Ekholm and Y.~Lekili.
\newblock Duality between {L}agrangian and {L}egendrian invariants.
\newblock {\em arXiv:1701.01284 (to appear in Geom. Topol.)}, 2017.

\bibitem{FH}
A.~Floer and H.~Hofer.
\newblock Symplectic homology. {I}. {O}pen sets in {${\bf C}^n$}.
\newblock {\em Math. Z.}, 215(1):37--88, 1994.

\bibitem{Friedman}
R.~Friedman.
\newblock Simultaneous resolution of threefold double points.
\newblock {\em Math. Ann.}, 274(4):671--689, 1986.

\bibitem{FutakiUedaSurvey}
M.~Futaki and K.~Ueda.
\newblock Homological mirror symmetry for {B}rieskorn-{P}ham singularities.
\newblock In {\em Proceeding of the 66th {J}apan {G}eometry {S}ymposium}, pages
  98--107. Saga {U}niversity, 2009.

\bibitem{FutakiUeda}
M.~Futaki and K.~Ueda.
\newblock Homological mirror symmetry for {B}rieskorn-{P}ham singularities.
\newblock {\em Selecta Math. (N.S.)}, 17(2):435--452, 2011.

\bibitem{FutakiUedaD}
M.~Futaki and K.~Ueda.
\newblock Homological mirror symmetry for singularities of type {D}.
\newblock {\em Math. Z.}, 273(3-4):633--652, 2013.

\bibitem{Gammage}
B.~Gammage.
\newblock Mirror symmetry for {B}erglund-{H}\"{u}bsch {M}ilnor fibers.
\newblock {\em arXiv:2010.15570}, 2020.

\bibitem{Ganatra}
S.~Ganatra.
\newblock Symplectic {C}ohomology and {D}uality for the {W}rapped {F}ukaya
  {C}ategory.
\newblock {\em arXiv:1304.7312}, 2013.

\bibitem{Habermann}
M.~Habermann.
\newblock Homological mirror symmetry for nodal stacky curves.
\newblock {\em arXiv:2101.12178}, 2021.

\bibitem{HabermannSmith}
M.~Habermann and J.~Smith.
\newblock Homological {B}erglund-{H}\"{u}bsch mirror symmetry for curve
  singularities.
\newblock {\em J. Symplectic Geom.}, 18(6):1515--1574, 2020.

\bibitem{H}
H.~Hofer.
\newblock Symplectic capacities.
\newblock In {\em Geometry of low-dimensional manifolds, 2 ({D}urham, 1989)},
  volume 151 of {\em London Math. Soc. Lecture Note Ser.}, pages 15--34.
  Cambridge Univ. Press, Cambridge, 1990.

\bibitem{Katz}
S.~Katz.
\newblock Small resolutions of {G}orenstein threefold singularities.
\newblock In {\em Algebraic geometry: {S}undance 1988}, volume 116 of {\em
  Contemp. Math.}, pages 61--70. Amer. Math. Soc., Providence, RI, 1991.

\bibitem{Keller2}
B.~Keller.
\newblock Derived invariance of higher structures on the {H}ochschild complex.
\newblock {\em Preprint, 2003}.

\bibitem{Keller}
B.~Keller.
\newblock Deformed {C}alabi-{Y}au completions.
\newblock {\em J. Reine Angew. Math.}, 654:125--180, 2011.
\newblock With an appendix by Michel Van den Bergh.

\bibitem{Laufer}
H.~B. Laufer.
\newblock On {${\bf C}P\sp{1}$} as an exceptional set.
\newblock In {\em Recent developments in several complex variables ({P}roc.
  {C}onf., {P}rinceton {U}niv., {P}rinceton, {N}. {J}., 1979)}, volume 100 of
  {\em Ann. of Math. Stud.}, pages 261--275. Princeton Univ. Press, Princeton,
  N.J., 1981.

\bibitem{LPascaleff}
Y.~Lekili and J.~Pascaleff.
\newblock Floer cohomology of {$\mathfrak{g}$}-equivariant {L}agrangian branes.
\newblock {\em Compos. Math.}, 152(5):1071--1110, 2016.

\bibitem{LU1}
Y.~Lekili and K.~Ueda.
\newblock Homological mirror symmetry for {M}ilnor fibers via moduli of
  {$A_\infty$}-structures.
\newblock {\em arXiv:1806.04345 (to appear in J. of Topology)}, 2018.

\bibitem{LU2}
Y.~Lekili and K.~Ueda.
\newblock Homological mirror symmetry for {M}ilnor fibers of simple
  singularities.
\newblock {\em Algebr. Geom.}, 8(5):562--586, 2021.

\bibitem{Markushevich}
D.~Markushevich.
\newblock Minimal discrepancy for a terminal c{DV} singularity is {$1$}.
\newblock {\em J. Math. Sci. Univ. Tokyo}, 3(2):445--456, 1996.

\bibitem{McLean}
M.~McLean.
\newblock Reeb orbits and the minimal discrepancy of an isolated singularity.
\newblock {\em Invent. Math.}, 204(2):505--594, 2016.

\bibitem{Milnor}
J.~Milnor.
\newblock {\em Singular points of complex hypersurfaces}.
\newblock Annals of Mathematics Studies, No. 61. Princeton University Press,
  Princeton, N.J.; University of Tokyo Press, Tokyo, 1968.

\bibitem{Mori}
S.~Mori.
\newblock On {$3$}-dimensional terminal singularities.
\newblock {\em Nagoya Math. J.}, 98:43--66, 1985.

\bibitem{Mumford}
D.~Mumford.
\newblock The topology of normal singularities of an algebraic surface and a
  criterion for simplicity.
\newblock {\em Inst. Hautes \'{E}tudes Sci. Publ. Math.}, (9):5--22, 1961.

\bibitem{Orlov}
D.~Orlov.
\newblock Formal completions and idempotent completions of triangulated
  categories of singularities.
\newblock {\em Adv. Math.}, 226(1):206--217, 2011.

\bibitem{Pinkham}
H.~C. Pinkham.
\newblock Factorization of birational maps in dimension {$3$}.
\newblock In {\em Singularities, {P}art 2 ({A}rcata, {C}alif., 1981)},
  volume~40 of {\em Proc. Sympos. Pure Math.}, pages 343--371. Amer. Math.
  Soc., Providence, RI, 1983.

\bibitem{PolishchukVarolgunes}
A.~Polishchuk and U.~Varolgunes.
\newblock On homological mirror symmetry for chain type polynomials.
\newblock {\em arXiv:2105.03808}, 2021.

\bibitem{Reid}
M.~Reid.
\newblock Minimal models of canonical {$3$}-folds.
\newblock In {\em Algebraic varieties and analytic varieties ({T}okyo, 1981)},
  volume~1 of {\em Adv. Stud. Pure Math.}, pages 131--180. North-Holland,
  Amsterdam, 1983.

\bibitem{Schwarz}
M.~Schwarz.
\newblock {\em Cohomology operations from \(S^1\)-cobordisms in Floer
  homology}.
\newblock 1995.
\newblock Thesis (Ph.D.)--ETH Z\"{u}rich.

\bibitem{seidel}
P.~Seidel.
\newblock A biased view of symplectic cohomology.
\newblock In {\em Current developments in mathematics, 2006}, pages 211--253.
  Int. Press, Somerville, MA, 2008.

\bibitem{SeidelSuspending}
P.~Seidel.
\newblock Suspending {L}efschetz fibrations, with an application to local
  mirror symmetry.
\newblock {\em Comm. Math. Phys.}, 297(2):515--528, 2010.

\bibitem{SeidelDisj}
P.~Seidel.
\newblock Disjoinable {L}agrangian spheres and dilations.
\newblock {\em Invent. Math.}, 197(2):299--359, 2014.

\bibitem{Tjurina}
G.~N. Tjurina.
\newblock Resolution of singularities of flat deformations of double rational
  points.
\newblock {\em Funkcional. Anal. i Prilo\v{z}en.}, 4(1):77--83, 1970.

\bibitem{Uebele}
P.~Uebele.
\newblock Symplectic homology of some {B}rieskorn manifolds.
\newblock {\em Math. Z.}, 283(1-2):243--274, 2016.

\bibitem{UebelePeriodic}
P.~Uebele.
\newblock Periodic {R}eeb flows and products in symplectic homology.
\newblock {\em J. Symplectic Geom.}, 17(4):1201--1250, 2019.

\bibitem{VanKoert}
O.~van Koert.
\newblock Contact homology of {B}rieskorn manifolds.
\newblock {\em Forum Math.}, 20(2):317--339, 2008.

\bibitem{Varchenko}
A.~N. Var\v{c}enko.
\newblock Contact structures and isolated singularities.
\newblock {\em Vestnik Moskov. Univ. Ser. I Mat. Mekh.}, (2):18--21, 101, 1980.

\bibitem{viterbo}
C.~Viterbo.
\newblock Functors and computations in {F}loer homology with applications. {I}.
\newblock {\em Geom. Funct. Anal.}, 9(5):985--1033, 1999.

\bibitem{Wall}
C.~T.~C. Wall.
\newblock Classification problems in differential topology. {VI}.
  {C}lassification of {$(s-1)$}-connected {$(2s+1)$}-manifolds.
\newblock {\em Topology}, 6:273--296, 1967.

\end{thebibliography}

\Addresses

\end{document}